\documentclass[11pt]{amsart}
\usepackage{amssymb,amsmath,amsfonts,amscd,euscript,verbatim,}

\usepackage{hyperref}
\newcommand{\nc}{\newcommand}

\usepackage{tikz-cd}

\numberwithin{equation}{section}
\newtheorem{thm}{Theorem}[section]
\newtheorem{prop}[thm]{Proposition}
\newtheorem{lem}[thm]{Lemma}
\newtheorem{cor}[thm]{Corollary}
\newtheorem{rem}[thm]{Remark}

\newtheorem{example}[thm]{Example}
\newtheorem{dfn}[thm]{Definition}

\newtheorem*{thma}{Theorem A}
\newtheorem*{thmb}{Theorem B}

\newtheorem*{cor*}{Corollary}

\nc{\cA}{\mathcal{A}}
\nc{\gl}{\mathfrak{gl}}
\nc{\GL}{\mathfrak{GL}}
\nc{\g}{\mathfrak{g}}
\nc{\gh}{\widehat\g}
\nc{\h}{\mathfrak{h}}
\nc{\la}{\lambda}
\nc{\al}{\alpha }
\nc{\be}{\beta }
\nc{\ve}{\varepsilon }
\nc{\om}{\omega }

\nc{\ch}{{\mathop {\rm ch}}}
\nc{\Tr}{{\mathop {\rm Tr}\,}}
\nc{\Id}{{\mathop {\rm Id}}}
\nc{\ad}{{\mathop {\rm ad}}}
\nc{\bra}{\langle}
\nc{\ket}{\rangle}
\nc{\pa}{\partial}
\nc{\ld}{\ldots}
\nc{\cd}{\cdots}
\nc{\hk}{\hookrightarrow}
\nc{\T}{\otimes}
\nc{\gr}{\mathrm{gr}}
\nc{\ov}{\overline}

\nc{\cO}{\mathcal O}
\nc{\cK}{\mathcal K}
\nc{\cP}{\mathcal P}
\nc{\cL}{\mathcal L}
\nc{\cF}{\mathcal F}
\nc{\cW}{\mathcal W}
\nc{\cD}{\mathcal D}
\nc{\msl}{\mathfrak{sl}}
\nc{\mgl}{\mathfrak{gl}}
\nc{\fQ}{\mathfrak{Q}}
\nc{\U}{\mathrm U}

\nc{\Spec}{\operatorname{Spec}}
\nc{\Proj}{\operatorname{Proj}}

\newcommand{\bC}{{\mathbb C}}

\newcommand{\bZ}{{\mathbb Z}}
\newcommand{\bN}{{\mathbb N}}
\newcommand{\bA}{{\mathbb A}}
\newcommand{\bP}{{\mathbb P}}

\newcommand{\fh}{{\mathfrak h}}
\newcommand{\fg}{{\mathfrak g}}
\newcommand{\fgh}{{\widehat{\mathfrak g}}}
\newcommand{\fb}{{\mathfrak b}}
\newcommand{\fn}{{\mathfrak n}}

\newcommand{\bc}{{\bf c}}
\newcommand{\bd}{{\bf d}}

\newcommand{\Gr}{{\rm Gr}}
\nc{\Q}{\mathfrak Q}
\newcommand{\bW}{{\mathbb{W}}}
\newcommand{\uG}{{\mathbb G}}
\newcommand{\ula}{{\underline{\la}}}
\newcommand{\umu}{{\underline{\mu}}}
\nc{\bD}{\mathbb{D}}

\nc{\on}{\operatorname} \nc\ol{\overline} \nc\ul{\underline}

\nc{\BA}{{\mathbb{A}}} \nc{\BC}{{\mathbb{C}}} \nc{\BF}{{\mathbb{F}}}
\nc{\BD}{{\mathbb{D}}} \nc{\BG}{{\mathbb{G}}} \nc{\BQ}{{\mathbb{Q}}}
\nc{\BM}{{\mathbb{M}}} \nc{\BN}{{\mathbb{N}}} \nc{\BO}{{\mathbb{O}}}
\nc{\BP}{{\mathbb{P}}} \nc{\BR}{{\mathbb{R}}}
\nc{\BZ}{{\mathbb{Z}}} \nc{\BS}{{\mathbb{S}}} \nc{\BW}{{\mathbb{W}}}

\nc{\CA}{{\mathcal{A}}} \nc{\CL}{{\mathcal{L}}} \nc{\CV}{{\mathcal{V}}} \nc{\CW}{{\mathcal{W}}}
\nc{\CalD}{{\mathcal{D}}}

\nc{\sic}{{\on{sc}}}
\nc{\add}{{\on{add}}}
\nc{\svee}{{\!\scriptscriptstyle\vee}}
\newcommand{\oA}{\vphantom{j^{X^2}}\smash{\overset{\circ}{\vphantom{\rule{0pt}{0.55em}}\smash{\BA}}}}

\begin{document}
	
	\title[Beilinson-Drinfeld Schubert varieties]
	{Beilinson-Drinfeld Schubert varieties and Global Demazure modules}
	
	\author{Ilya Dumanski}
	\address{Ilya Dumanski:\newline
		Department of Mathematics, National Research University Higher School of Economics, Russian Federation,
		Usacheva str. 6, 119048, Moscow, \newline
		{\it and }\newline
		Independent University of Moscow, 119002, Bolshoy Vlasyevskiy Pereulok 11, Russia, Moscow.
	}
	\email{ilyadumnsk@gmail.com}

	\author{Evgeny Feigin}
	\address{Evgeny Feigin:\newline
		Department of Mathematics, National Research University Higher School of Economics, Russian Federation,
		Usacheva str. 6, 119048, Moscow,\newline
		{\it and }\newline
		Skolkovo Institute of Science and Technology, Skolkovo Innovation Center, Building 3,
		Moscow 143026, Russia.
	}
	\email{evgfeig@gmail.com}
	
	\author{Michael Finkelberg}
	\address{Michael Finkelberg:\newline
		Department of Mathematics, National Research University Higher School of Economics, Russian Federation,
		Usacheva str. 6, 119048, Moscow,\newline
		Skolkovo Institute of Science and Technology, Skolkovo Innovation Center, Building 3,
		Moscow 143026, Russia,
		{\it and }\newline
		Institute for Information Transmission Problems of RAS, Moscow, Russia.
	}
	\email{fnklberg@gmail.com}
	
	\begin{abstract}
		We compute the spaces of sections of powers of the determinant line bundle on the spherical
		Schubert subvarieties of the Beilinson-Drinfeld affine Grassmannians. The answer is given in terms of global Demazure modules over the current Lie algebra.
	\end{abstract}

	\maketitle
	
	\section*{Introduction}
	Let $\fg$ be a simple complex Lie algebra. To simplify the notation, in the introduction we assume that $\fg$ is simply-laced. 
	We drop this restriction in the main body of the paper.
	
	The central objects of the algebraic representation theory of $\fg$
	are finite-dimensional irreducible representations $V_\la$ of $\fg$ labeled by the dominant integral weights $\la\in P_+$.
	The geometric objects responsible for these representations are the flag varieties. In particular, 
	flag varieties are naturally embedded into projectivizations of irreducible $\fg$-modules and  
	the celebrated  Borel-Weil theorem states that finite-dimensional 
	$\fg$ modules are realized as (dual) spaces of sections of line bundles on flag varieties (see e.g. \cite{Fu,Kum}).
	These properties are still valid after passing to the Demazure submodules inside $V_\la$ and to the Schubert subvarieties in flag varieties.      
	
	In this paper we are interested in the representation theory (algebraic and geometric) of two natural infinite-dimensional analogues 
	of the Lie algebra $\fg$ -- the current algebra $\fg[t]=\fg\T\bC[t]$ and the (untwisted) affine Kac-Moody Lie algebra $\gh$ with the
	natural embedding $\fg[t]\subset \gh$. The $\gh$-analogues of the $\fg$-modules $V_\la$ are (infinite-dimensional) integrable highest 
	weight representations $L(\Lambda)$ (see \cite{Kac}). The central element of $\gh$ acts on $L(\Lambda)$ by a constant called the level of representation. In particular, there are finitely many level one integrable modules $L(\Lambda_i)$, $i=0,\dots,m$, where $L(\Lambda_0)$
	is the basic representation. In this paper we
	will only consider modules $L(\ell\Lambda_i)$ for $\ell\in\bZ_{>0}$. The projectivization $\bP(L(\Lambda_i))$
	contains a partial affine flag variety $\Gr(\Lambda_i)$ as the closure of the orbit of the highest weight line with respect to the action
	of the affine Kac-Moody group (\cite{Kum}). The disjoint union $\sqcup_{i=0}^m \Gr(\Lambda_i)$ is isomorphic to the affine 
	Grassmannian $\Gr$ for the adjoint Lie group of the Lie algebra $\fg$. 
	
	The Demazure submodules in integrable representations $L(\Lambda)$ are labeled by the elements of the extended affine Weyl group.
	We will only consider the $\fg[t]$-invariant Demazure modules inside $L(\ell\Lambda_i)$ (note that in general a Demazure module
	is only acted upon by the Iwahori subalgebra that is strictly contained in the current algebra). In particular, the $\fg[t]$-invariant
	Demazure modules $D_{1,\la}$ inside the level one integrable representations are labeled by the dominant integral weights $\la$. We denote
	by $D_{\ell,\la}$, $\ell\ge 1$ the level $\ell$ affine Demazure modules contained in the $\ell$-th tensor power of level one module $D_{1,\la}$.
	The projectivized Demazure module $\bP(D_{1,\la})$ contains the spherical Schubert variety $\ol\Gr{}^\la$ as the orbit closure of 
	the current group action.
	Thanks to the embedding $\ol\Gr{}^\la\subset \bP(D_{1,\la})$ the Schubert varieties are equipped with ample line bundle $\cL$, such that
	the dual space of sections of  $\cL^{\T\ell}$ is isomorphic to $D_{\ell,\la}$ for any $\ell$. The line bundle $\cL$ on a Schubert variety can be also obtained
	as the restriction of the determinant line bundle on the affine Grassmannian (see \cite{Kum,Z2}). 
	
	The current algebra $\fg[t]$ possesses a remarkable family of cyclic finite-dimensional modules $W_\la$ called the local Weyl modules
	(see \cite{CL,CP,FL2,KN,Naoi}). In particular, as $\fg$-module, $W_\la$ is isomorphic to the tensor product of fundamental local Weyl modules,
	where the number of factors of the form $W_{\om}$ is exactly the coefficient of $\om$ in the decomposition of $\la$. 
	We note that in the simply-laced case one has an isomorphism $W_\la\simeq D_{1,\la}$.
	The global 
	Weyl modules $\bW_\la$ are infinite-dimensional cyclic representations of $\fg[t]$ (see \cite{BF1,CFK,CI,FeMa1, Kato1}). One of the most important properties of the global Weyl modules is the existence of free action of the commutative highest weight algebra $\cA_\la$,
	commuting with the $\fg[t]$-action. In particular, one obtains a family of (finite-dimensional) $\fg[t]$-modules, labeled by the closed points in $\Spec(\cA_\la)$, obtained as fibers of $\bW_\la$ with respect to $\cA_\la$; the local Weyl module is the fiber at the origin. 
	
	A generalization of this picture was suggested in \cite{DF}. The authors introduced a family of
	cyclic (infinite-dimensional)
	global Demazure modules $\bD(\ell,\ula)$ (that were denoted $R(D_{\ell,\la_1},\dots,D_{\ell,\la_k})$
	in~\cite{DF}) corresponding to a collection 
	of dominant integral nonzero weights $\ula\in P_+^k$ and an  integer $\ell>0$;
	in particular, if all $\la_i$ are fundamental and $\ell=1$,
	then one gets back the global Weyl module 
	(this is no longer true in the non simply-laced case).
	The global Demazure modules arise naturally in connection with the study of the projective arc spaces (see \cite{Mu1,Mu2,Nash}).
	The modules $\bD(\ell,\ula)$ are acted upon by a commutative (highest weight) algebra $\cA(\ula)=\cA(\la_1,\dots,\la_k)$ whose action commutes 
	with the $\fg[t]$ action (see \cite{BCES,EGL,KMSV,SV} for the examples of similar algebras). 
	The spectrum $\Spec(\cA(\ula))$ is
	the closure of a stratum of the diagonal stratification of a colored configuration space of the
	affine line
	(see section \ref{high wei alg} for precise definitions). In particular, a closed point
	$\bc\in \bA^k$ defines the same named closed point in $\Spec(\cA(\ula))$.
	For a point $\bc\in\Spec(\cA(\ula))$ we denote by $\bD(\ell,\ula)_\bc$ the fiber of the global Demazure module at $\bc$. 
	Our first theorem is as follows:
	\begin{thma}
	  \textup{(a)} Assume that $\lambda_i\ne0$ for all $i$.
          One has an isomorphism of $\fg[t]$-modules
		\[
		\bD(\ell, \ula)_0 \simeq D_{\ell,\la_1+\dots+\la_k}.
		\]
		\textup{(b)} Let $\ula\in P_+^k$, $\underline\mu\in P_+^l$. If $\bc\in\bA^{k}$ and $\bd\in \bA^l$ have no common entries, 
		then the following factorization property (an isomorphism of $\fg[t]$-modules) holds:
		\[
		\bD(\ell, \ula\sqcup\umu)_{(\bc,\bd)}\simeq \bD(\ell, \ula)_{\bc}\T \bD(\ell, \umu)_{\bd}.
		\]
		\textup{(c)} The global Demazure module $\bD(\ell, \ula)$ is free over $\cA(\ula)$.\\
		\textup{(d)} The direct sum of $\cA(\ula)$-dual modules $\bigoplus_{\ell\ge 0} \bD(\ell,\ula)^\vee$ carries a natural structure of 
		$\cA(\ula)$-algebra. 
	\end{thma}
	
	The properties of the global Demazure modules collected in Theorem A are parallel to the
	properties of the Beilinson-Drinfeld spherical
	Schubert varieties over the affine line.
	The main goal of this paper is to describe this relation explicitly.
	More precisely, we:
	\begin{itemize}   
		\item identify the projective spectrum of the algebra $\bigoplus_{\ell\ge 0} \bD(\ell,\ula)^\vee$
		with the partially symmetrized BD spherical Schubert varieties;
		\item embed symmetrized BD spherical Schubert varieties into the projectivization of the  vector bundle $\cD(\ell, \ula)$ obtained as the localization of the (free) 
		$\cA(\ula)$-module $\bD(\ell, \ula)$;
		\item identify the dual sections of the determinant line bundle on symmetrized BD spherical Schubert
		varieties with global Demazure modules. 
	\end{itemize}
	
	Let us state our results in more detail. 
	Recall that the Beilinson-Drinfeld Grassmannians (BD Grassmannians for short) are global versions of
	the affine Grassmannians~\cite{BD1,BD2,FBZ,Z2} defined over the powers of an algebraic curve $X$; 
	in this paper we only consider
	the case $X=\bA^1$ and denote the corresponding BD Grassmannians by $\Gr_{\bA^k}$ 
	(see e.g. \cite{BKK,CK,CW,Kam,MVy} for various applications in geometric representation theory).
	The Grassmannians $\Gr_{\bA^k}$ are ind-varieties over the configuration space $\bA^k$,
	and the ind-structure is provided by the BD spherical Schubert varieties  $\ol\Gr{}^{\ula}$,
	labeled by $k$-tuples of dominant coweights $\ula=(\la_1,\dots,\la_k)\in P_+^k$. A group scheme $\uG(k)$ over $\bA^k$
	(the global analogue of the current group $G(\bC[\![t]\!])$) acts on $\Gr_{\bA^k}$ fiberwise,
	and the BD spherical Schubert varieties are the closures of orbits
	of $\uG(k)$ in the generic fiber of $\Gr_{\bA^k}$
	(we note that the same group scheme acts on $\bP(\cD(\ell, \ula))$).
	The fibers of the projection 
	$\ol\Gr{}^{\ula}\to\bA^k$ are products of the spherical Schubert 
	subvarieties of the affine Grassmannian (this is a manifestation of the crucial factorization
	property of the BD Grassmannians).
	
	BD Grassmannians carry the ample determinant line bundle $\cL$; we keep the same notation for the restriction of this line bundle to the BD spherical Schubert
	varieties. 
	The space of sections $H^0(\ol\Gr{}^{\ula},\cL^{\T\ell})$ 
	is naturally a  $\fg[t]-\bC[\bA^k]$-bimodule (we note that the higher cohomology
	$H^{>0}(\ol\Gr{}^{\ula},\cL^{\T\ell})$ vanish). However,
	as a module over the current algebra it is not cyclic and hence hard to describe. In order to resolve this problem we consider
	a partially symmetrized version $\ol\Gr{}^{(\ula)}$ of the BD spherical Schubert varieties
	(see~Section \ref{bdg} for precise definition).
	The variety $\ol\Gr{}^{(\ula)}$ is equipped with a natural projection onto $\Spec(\cA(\ula))$ and the determinant line bundle descends 
	to the symmetrized BD spherical Schubert varieties. We prove the following theorem.
	\begin{thmb}
		For all $\ell\ge 1$ and $\ula = (\la_1, \hdots, \la_k) \in P_+^k$, such that all $\la_i$ are nonzero, one has:\\
		\textup{(a)} a $\uG(k)$-equivariant embedding:
		\[
		\ol\Gr{}^{(\ula)}\subset \bP(\cD(\ell, \ula));
		\]
		\textup{(b)} an isomorphism of $\Spec(\cA(\ula))$-schemes:
		\[
		\ol\Gr{}^{(\ula)}\simeq \Proj(\bigoplus_{\ell\ge 0} \bD(\ell,\ula)^\vee);
		\]
		\textup{(c)}
		an isomorphism of $\fg[t]-\cA(\ula)$-bimodules:
		\[
		H^0(\Gr^{(\ula)},\cL^{\T\ell}) \simeq \bD(\ell,\ula)^\vee;
		\]
		\textup{(d)}
		an isomorphism of $\fg[t]-\bC[\bA^k]$-bimodules:
		\[
		H^0(\Gr^{\ula},\cL^{\T\ell}) \simeq \bD(\ell,\ula)^\vee \T_{\cA(\ula)}\bC[\bA^k],
		\]
		where $M^\vee$ stands for the $\cA(\ula)$-dual to an $\cA(\ula)$-module $M$.
		
	\end{thmb}
	Let us close with the following remark. In the main body of the paper, we denote the weights and roots of $\fg$ by the checked letters 
	(like $\la^\svee$ and $\al^\svee$) and reserve the non-checked notation for the dual data (coroots and coweights). The reason is
	that the central role in our paper is played by the spherical Schubert varieties in the affine (Beilinson-Drinfeld) Grassmannians.
	These varieties are naturally labeled by the coweights (rather than weights), which explains our choice of notation.
	Note that in the ADE case all the checks can be removed without any harm.
	Also, in the simply laced case, if $\lambda_1,\ldots,\lambda_k$ are all fundamental, and
	$\lambda=\lambda_1+\ldots+\lambda_k$, the global Demazure modules $\bD(1,\ula)$ are nothing but
	the global Weyl modules $\bW_\lambda$. However, if $\fg$ is not simply laced, there is no such
	coincidence anymore. That is why we choose to call $\bD(\ell,\ula)$ global Demazure modules
	as opposed to ``higher-level global Weyl modules''.
	
	Our paper is organized as follows. 
	In Section \ref{Generalities} we collect notation and recall main definitions. 
	In Section \ref{bdg} we introduce the symmetrized version of the Beilinson-Drinfeld Grassmannians and Schubert varieties over 
	the spectrum of the highest weight algebras. 
	In Section \ref{Global modules} we study the properties of the global Demazure modules. In particular,
	we prove that they are free over the highest weight algebras. 
	In Section \ref{GDm and BD} we compute the spaces of sections of the powers of the determinant line bundle on BD Schubert varieties. 
	In Appendix, we discuss a connection between global modules and the associativity of the fusion product. We also collect the key objects of the paper.

	\section*{Acknowledgments}
	The authors would like to thank Syu Kato for his suggestion to study the algebra of dual global
	Demazure modules. We are indebted to Roman Travkin for his careful reading, corrections, and improvements of the first draft. We are also grateful to Alexander Braverman and Boris Feigin for
	useful discussions. Finally, our thanks go to the anonymous referee for a very careful reading
	of the manuscript and many useful suggestions.
	The work was partially supported by the grant RSF 19-11-00056.
	I.D.\ is supported in part by the Simons Foundation.

	\section{Generalities}\label{Generalities}
	We start with describing the notation for the key objects of the paper.
	
	\subsection{Classical objects}
	Let $\fg$ be a simple Lie algebra over $\BC$. The corresponding simply connected (resp.\ adjoint)
	complex Lie group will be denoted $G^\sic$ (resp.\ $G^\ad$).
	Let $\fg=\fn_+\oplus\fh\oplus \fn_-$ be the Cartan decomposition and let $r=\dim\fh$
	be the rank of $\fg$. We denote by $\om_1^\svee,\dots,\om_r^\svee$ the fundamental weights of $\fg$
	and by $\al_1^\svee,\dots,\al_r^\svee$ its simple roots. Let $P^\vee=\bigoplus_{i=1}^r \bZ\om_i^\svee$
	be the weight lattice of $G^\sic$ containing the root lattice $Q^\vee=\bigoplus_{i=1}^r \bZ\al_i^\svee$
	(that coincides with the weight lattice
	of $G^\ad$). Let $P_+^\vee=\bigoplus_{i=1}^r \bZ_{\ge 0}\om_i^\svee\subset P^\vee$
	be the set of dominant integral weights. 
	Given $\la^\svee=\sum_{i=1}^r m_i\om_i^\svee\in P_+^\vee$ we set $|\la^\svee|=\sum_{i=1}^r m_i$. 
	
	For a weight $\la^\svee\in P_+^\vee$, let $V_{\la^\svee}$ be the highest weight $\la^\svee$ irreducible $\fg$-module; in particular, the highest weight vector of $V_{\la^\svee}$ is of the $\fh$ weight $\la^\svee$ and is killed by $\fn_+$.
	Let $W$ be the (finite) Weyl group of $\fg$ with the longest element $w_0$. In particular, the lowest weight vector
	in $V_{\la^\svee}$ is of weight $w_0\la^\svee$.
	
	We denote by $P=P_\ad$ (resp.\ $Q=P_\sic$) the coweight lattice of $G^\ad$ (resp.\ of $G^\sic$).
	Thus we have perfect pairings $P_\ad\times Q^\vee\to\BZ,\ P_\sic\times P^\vee\to\BZ$.
	The minimal invariant integral bilinear form on $P_\sic$ (such that the
	square length of a short coroot is 2) gives rise to a linear map $\iota\colon P_\sic\to Q^\vee$.
	It extends by linearity to the same named map
	$P_\sic\subset P_\ad\stackrel{\iota}{\longrightarrow}Q^\vee\otimes_\BZ\BQ$, and
	$\iota(P_\ad)\subset P^\vee\subset Q^\vee\otimes_\BZ\BQ$. The resulting map $P=P_\ad\to P^\vee$ will be
	also denoted by $\iota$. In the simply laced case $\iota\colon P\to P^\vee$ is an isomorphism.
	
	The fundamental coweights in $P$ are denoted $\omega_1,\ldots,\omega_r$; the simple coroots
	in $P$ are denoted by $\alpha_1,\ldots\alpha_r$. We set
	$P_+=\bigoplus_{i=1}^r\BZ_{\ge 0}\omega_i\subset P$.
	
	\subsection{Current algebra modules}
	\label{curr alg mod}
	Let $\fg[t]=\fg\T\bC[t]$ be the current algebra of $\fg$. In what follows we consider graded $\fg[t]$-modules $M$, i.e.
	$M=\bigoplus_{i\ge 0} M_i$, $\fg\T t^k\colon  M_i \to M_{i+k}$. If all $M_i$ are finite-dimensional, then the graded character
	$\ch_q(M)$ is a generating function $\sum_{i\ge 0} q^i\ch M_i$ of the characters of the $\fg\T 1$-modules. 
	
	A module $M$ is called cyclic if it is  generated by a single vector. The cyclic product of two cyclic $\fg[t]$-modules 
	$M_1$ and $M_2$ with fixed cyclic vectors $w_1\in M_1$ and $w_2\in M_2$ is defined as
	\[
	M_1\odot M_2=\U(\fg[t]) . w_1\T w_2\subset M_1\T M_2.
	\]
	For a collection of pairwise distinct complex numbers $c_1,\dots,c_n$ and cyclic graded $\fg[t]$-modules 
	$M_1,\dots,M_n$ the module $M_1(c_1) \T \dots \T M_n(c_n)$ is known to be cyclic with the cyclic vector being the 
	tensor product of cyclic vectors of $M_i$. Here a $\fg[t]$-module $M_i(c_i)$ is defined to be isomorphic to $M_i$ as a vector space and 
	the action of the current algebra on it is twisted by the automorphism $x\T t^i \mapsto x\T (t-c)^i$. We note that if one starts with
	graded cyclic modules $M_i$, the tensor product $\bigotimes_{i=1}^n M_i(c_i)$ is not graded in general.
	The fusion product (graded tensor product) $M_1 \ast \hdots \ast M_k$ is defined as the associated graded of $\bigotimes_{i=1}^n M_i(c_i)$
	with respect to the filtration induced by the action of the ($t$-degree graded) universal enveloping algebra $\U(\fg[t])$
	on the tensor product of cyclic vectors of $M_i$  (\cite{FeLo}).
	
	Let $\bW_{\la^\svee}$  and $W_{\la^\svee}$ be the global and local Weyl modules of highest weight $\la^\svee$ over the Lie algebra $\fg[t]$ (see \cite{CP, CL, FL2, Kato1, Naoi}).
	Let $D_{\lambda}$ be the level one
	affine Demazure module with highest weight $\iota(\lambda)$; in particular, for simply laced algebras
	$\iota$ is an isomorphism and $W_{\iota(\la)}\simeq D_{\la}$ for any coweight $\la$. 
	For $\lambda\in P_+$ we denote
	$\BW_{\iota(\lambda)}$ (resp.\ $W_{\iota(\lambda)}$) simply by $\BW_\lambda$
	(resp.\ $W_\lambda$).
	
	For $\ell\in\bZ_{>0}$ we denote by $D_{\ell,\la}$ the level $\ell$ affine Demazure module with highest weight
	$\ell\iota(\lambda)$ (see subsection \ref{affine and demazure} for details).
	
	Let $M_1,\dots,M_k$ be graded cyclic $\fg[t]$-modules with cyclic vectors $w_i$ of dominant nonzero weights such that $t\h[t]$ annihilates their cyclic vectors.
	Then we define the global module (see \cite{DF})
	\[
	R(M_1,\dots,M_k)=M_1[t]\odot\cdots\odot M_k[t],
	\]
	where $M_i[t]$ is defined as a module isomorphic to $M_i\T\bC[t]$ as a vector space with the action
	of $\fg[t]$ given by
	\begin{equation}\label{formula}
	xt^l.v \otimes t^k=\sum_{j=0}^l (-1)^{l-j}\binom{l}{j}(xt^j.v)\otimes t^{l+k-j} 
	\end{equation}
	for $l, k\in \bZ_{\ge 0}$, $x \in \fg$, $v \in M_i$.
	\begin{rem}
		{\em The analogous formula used in \cite{DF,FeMa2} has no sign $(-1)^{l-j}$. We introduce it here in order to
			match the formulas in the Beilinson-Drinfeld setup, where the minus sign pops up via the change of coordinates
			$t \mapsto t-x$. This sign change obviously produces no harm (simply changing $t \mapsto -t$ in the current algebra parametrization).} 
	\end{rem}
	
	\begin{rem}
	 The modules $R(M_1,\dots,M_k)$ do depend on the choice of the cyclic vectors $w_i$ of $M_i$.  
	\end{rem}
	
	The global module $R(M_1,\dots,M_k)$ admits the right action of $\U(\h[t])$, which commutes with the $\g[t]$-action. The highest weight algebra is defined as a quotient of $\U(\h[t])$ by the annihilator of the cyclic vector $\otimes_{i=1}^k w_i$ of $R(M_1,\dots,M_k)$. It turns out that the highest weight algebra depends only on the weights of cyclic vectors $w_i$ of $M_i$ (not on a particular choice of modules).
	If the weight of $w_i$ is $\iota(\lambda_i)$, then
	we denote the highest weight algebra of $R(M_1,\dots,M_k)$ by $\cA(\la_1,\dots,\la_k)$.
	We will use a short-hand notation
	$\cA(\ula)=\cA(\la_1,\dots,\la_k)$, where $\ula=(\la_1,\dots,\la_k)$.
	
	Since the weight $\iota(\lambda_i)$ subspace of a module $M_i[t]$ is isomorphic to a polynomial algebra in one variable, the algebra $\cA(\ula)$ is naturally embedded into $\bigotimes_{i=1}^k \cA(\la_i) \simeq \bC[z_1,\dots,z_k]$. More precisely, $\cA(\ula)$ is isomorphic to the subalgebra of the polynomial algebra
	$\bC[z_1,\dots,z_k]$ generated by the polynomials
	\[
	\langle\iota(\la_1),h\rangle z_1^l+\langle\iota(\la_2),h\rangle z_2^l+ \dots +
	\langle\iota(\la_k),h\rangle z_k^l,\ l\ge 1, h\in\fh.
	\]
	Indeed, for $h\in\fh$ and $l>0$ formula \eqref{formula} gives
	\[
	ht^l.\otimes_{i=1}^k w_i = \sum_{i=1}^k \otimes_{j=1}^{i-1} w_i
	\otimes  \langle\iota(\la_j),h\rangle w_jt^l \otimes_{j=i+1}^k w_j.
	\]
	In particular, if all $\la_i$ are fundamental coweights, $m_i=\#\{j: \la_j=\om_i\}$, and
	$\lambda=\sum_{j=1}^k\lambda_j=\sum_{i=1}^rm_i\omega_i$, then
	\[
	\cA(\ula)\cong \bC[z_1, \hdots , z_k]^{S_{m_1} \times \hdots \times S_{m_r}}=:\cA_\lambda.
	\] 
	
	We note that
	\begin{itemize}
		\item $\cA(\la_1,\dots,\la_k)\simeq \cA(\ell\la_1,\dots,\ell\la_k)$ for any $\ell\in \bN$.
		\item If $\fg$ is simply laced, all weights $\la_1,\dots,\la_k$ are fundamental, and
		$\la=\sum_{i=1}^k \la_i$, then 
		\[
		R(D_{\la_1},\dots,D_{\la_k})\simeq \bW_\la.
		\]
	\end{itemize}
	
	If all the coweights $\la_i$ are fundamental, and they sum up to $\lambda$,
	then we denote by $\bD_{\ell,\la}$ (global Demazure module) the module
	$R(D_{\ell,\la_1},\dots,D_{\ell,\la_k})$. In particular, for simply-laced $\g$ one has $\bD_{1,\la}\simeq \bW_\la$.
	
	If all the coweights $\la_i$ are fundamental, then the algebra $\cA(\la_1,\dots,\la_k)$ acts freely on 
	$R(D_{\la_1},\dots,D_{\la_k})$ and the fiber at the origin of the global Demazure module is isomorphic to the fusion product $D_{\la_1}*\dots* D_{\la_k} \simeq D_{\la_1 + \hdots + \la_k}$. The higher level analogue still holds with fundamental $\la_i$ replaced by $\ell\la_i$,
	see \cite{DF}.
	
	As we will prove in Proposition \ref{global demazure is projective}, for arbitrary dominant coweights $\lambda_1,\ldots,\lambda_k$,
	the module $R(D_{\ell,\la_1},\dots,D_{\ell,\la_k})$ is free over $\cA(\la_1,\dots,\la_k)$.
	We use the notation 
	\[
	\bD(\ell, \ula)=R(D_{\ell,\la_1},\dots,D_{\ell,\la_k}).
	\]
	\begin{rem}
		{\em We note that $\bD_{\ell,\la} = \bD(\ell,  \underbrace{\om_1, \hdots , \om_1}_{m_1}, \hdots ,
			\underbrace{\om_r, \hdots , \om_r}_{m_r})$ for a coweight $\lambda=\sum_{j=1}^rm_j\omega_j$.}
	\end{rem}
	
	In what follows we will need the following $\cA(\ula)$-analog of the cyclic power.
	Namely, let  
	\[
	\underbrace{R(M_1,\ldots,M_k)\odot_{\CA(\ula)}\ldots\odot_{\CA(\ula)}R(M_1,\ldots,M_k)}_\ell
	\]
	be the $\U(\fg[t])$-span of the $\ell$-th tensor power of the cyclic (highest weight) vector
	of $R(M_1,\ldots,M_k)$ inside the $\ell$-th tensor power over $\cA(\ula)$ of the module $R(M_1,\ldots,M_k)$.
	We denote this cyclic tensor power by $R(M_1,\ldots,M_k)^{\odot\ell}_{\CA(\ula)}$.
	This object will be important in Proposition~\ref{tensor product over A}.
	
	\subsection{Affine Lie algebras and Demazure modules} \label{affine and demazure}
	The details on the material below can be found in \cite{Kac,Kum}. 
	
	Let $\gh=\fg\T \bC[t,t^{-1}]\oplus \bC K\oplus\bC d$ be the untwisted affine Kac-Moody Lie algebra attached to $\fg$. 
	Here $K$ is central element
	and $d$ is (negated) degree operator (i.e. $[d,x\T t^i]=-ix\T t^i$). The algebra $\gh$ enjoys the Cartan decomposition
	$\gh=\fn_+^a\oplus\fh^a\oplus\fn_-^a$, where $\fh^a=\fh\T 1\oplus\bC K\oplus\bC d$ and 
	$\fn_+^a=\fg\T t\bC[t] \oplus \fn_+ \T 1$. We denote by $\fb^a= \fh^a \oplus \fn_+^a$ the Iwahori subalgebra. 
	
	Let $\Lambda_0^\vee$ be the level one basic integrable weight of $\gh$
	(in particular, $\Lambda_0^\vee(\fh\T 1)=0$). We also denote by $\Lambda_i^\vee$, $i=0,\dots,m$,
	the set of all integrable
	level one weights of $\gh$ and by $L(\Lambda_i^\vee)$ the corresponding highest weight $\gh$-modules.
	In particular, the number $m$ of the level one modules is equal to the cardinality of
	$P/Q \simeq \pi_1(G^\ad)$.

	Let $\Gr(\Lambda_i^\vee)\subset \bP(L(\Lambda_i^\vee))$, $i=0,\dots,m$, be the partial affine flag varieties corresponding 
	to maximal parabolic subgroups of the affine Kac-Moody group $\widehat{G}{}^\sic$
	(i.e.\ $\Gr(\Lambda_i^\vee)\simeq \widehat{G}{}^\sic/P_i$, where
	$P_i$ is the stabilizer of the highest weight line in $\bP(L(\Lambda_i^\vee))$). By the very definition, each 
	$\Gr(\Lambda_i^\vee)$ is equipped with the very ample line bundle $\cL$ (the pullback of $\cO(1)$ from $\bP(L(\Lambda_i^\vee))$)
	and one has the affine analog of the Borel-Weil theorem
	\[
	H^0(\Gr(\Lambda_i^\vee),\cL^{\T\ell})^*\simeq L(\ell\Lambda_i^\vee),\ \ell\ge 1, 
	\]
	where $L(\ell\Lambda_i^\vee)$ is the weight $\ell\Lambda_i^\vee$ integrable (level $\ell$) $\gh$-module and the superscript star denotes the restricted dual space.
	
	\begin{rem}
		{\em The union $\sqcup_{i=0}^m \Gr(\Lambda_i^\vee)$ is isomorphic to the affine Grassmannian of $G^\ad$, see the details below.}
	\end{rem}

	Let $W^a=W\ltimes P$ be the extended affine Weyl group (recall that $P$ is the coweight lattice
	of $G^\ad$). Then for any 
	$\la\in P_+$ there exists an element $w_\la\in W^a$ such that the $\fh$-weight of $w_\la\Lambda_0^\vee$
	is equal to $w_0\iota(\la)$.
	Let $\Lambda_i^\vee$ be the unique integrable  level one weight such that
	$w_\la\Lambda_0^\vee-\Lambda_i^\vee$ belongs to the root lattice of $\fg$. 
	Let $u_{w_0\iota(\la)}\in L(\Lambda_i^\vee)$ be a weight $w_0\iota(\la)$ vector. We define the Demazure module 
	$D_{1,\la} \subset L(\Lambda_i^\vee)$ as the $\U(\fb^a)$ span of the vector $u_{w_0\iota(\la)}$. An important property of the 
	Demazure modules $D_{1,\la}$ is that they are invariant with respect to the whole current algebra $\fg[t]\supset \fb^a$. In particular, $D_{1,\la}$ contains the irreducible $\fg$-module $V_{\iota(\la)}$ as the $\U(\fn_+)$ span of $u_{w_0\iota(\la)}$.
	
	The level $\ell$ Demazure module $D_{\ell,\la}$ is defined as the $\U(\fb^a)$ span of the vector
	$u_{w_0\iota(\la)}^{\T\ell}$. By definition,
	$D_{\ell,\la}$ is a subspace of $L(\Lambda_i^\vee)^{\T\ell}$. However, it is easy to see that 
	\[
	D_{\ell,\la}\subset L(\ell\Lambda_i^\vee)\subset L(\Lambda_i^\vee)^{\T\ell}.
	\] 
	By definition, one gets a natural structure of algebra on the space $D_{\bullet,\la}^*=\bigoplus_{\ell\ge 0} D_{\ell,\la}^*$ generated by the
	degree one component $D_{1,\la}^*$ (we set $D_{0,\la}=\bC$).  
	
	
	We define a spherical Schubert variety $\ol\Gr{}^\la$ as the closure of the $G^\sic(\cO)$-orbit of the
	line containing the lowest weight
	vector $u_{w_0\iota(\la)}$ (here $\cO=\bC[\![t]\!]$). Then $\ol\Gr{}^\la$ is embedded as a closed
	subscheme into the projectivization $\bP(D_{1,\la})$ of the Demazure module $D_{1,\la}$
	(see e.g.~\cite[Th\'eor\`eme 2.$\Sigma$ of Chapter X]{matas} or~\cite{kumin}). 
	Moreover, $\ol\Gr{}^\la$ is also embedded as a closed subscheme into the projectivization of an
	arbitrary level Demazure module $D_{\ell,\la}$ as the closure of the lowest weight line.  
	
	\begin{rem}\label{tla}
		{\em Let $t^\la\in\ol\Gr{}^\la\subset \bP(D_{1,\la})$ be the point corresponding to the weight
			$\iota(\la)$ line. Then $\ol\Gr{}^\la$ is the closure of the $G^\sic(\cO)$-orbit of $t^\la$.}  
	\end{rem}
	
	The embedding $\ol\Gr{}^\la\subset\bP(D_{1,\la})$ endows $\ol\Gr{}^\la$ with a very ample line bundle $\cL$,
	the pullback of $\cO(1)$. 
	The line bundle $\cL$ is a generator of the Picard group of $\ol\Gr{}^\la$, and one has the isomorphism of $\fg[t]$-modules:
	\[
	H^0(\ol\Gr{}^\la,\cL^{\T\ell})^* \simeq D_{\ell,\la} \text{ for all }\ell\ge 1.
	\]
	We obtain a presentation of $\ol\Gr{}^\la$ as the projective spectrum of the algebra of dual Demazure modules, i.e.
	$\ol\Gr{}^\la\simeq \Proj(\bigoplus_{\ell\ge 0} D_{\ell,\la}^*)$.
	
	We have 
	\[
	\Gr(\Lambda_i^\vee)=\bigcup_{\lambda : \Lambda_i^\vee-\iota(\la)\in Q^\vee+\Lambda_0^\vee}\ol\Gr{}^\la.
	\]
	Also, $\ol\Gr{}^\la\subset\ol\Gr{}^\mu$ if and only if
	$\mu-\la\in \bigoplus_{j=1}^r \bZ_{\ge 0}\al_j$. 
	
	\subsection{Affine Grassmannians}
	The affine Grassmannian of $G^\ad$ is $\Gr:=\Gr_{G^\ad}=G^\ad(\cK)/G^\ad(\cO)$, where $\cK=\bC(\!(t)\!)$ 
	is the ring of Laurent series and $\cO=\bC[\![t]\!]$ is the ring of Taylor series. The following properties of $\Gr$ can be found in
	\cite{Z1,Z2,Kum}.
	\begin{itemize}
		\item The connected components of $\Gr$ are in bijection with $P/Q$, i.e.\
		$\pi_0(\Gr)\simeq \pi_1(G^\ad)$.
		\item $\Gr = \sqcup_{i=0}^m \Gr(\Lambda_i^\vee)$.
		\item For any $i=0,\ldots,m,\ {\rm Pic}(\Gr(\Lambda_i^\vee))$ is generated by the class of the
		ample determinant line bundle $\cL$.  
	\end{itemize}
	
	Recall (see, for example, \cite{BL}) that the affine Grassmannian $\Gr$ is the moduli space of pairs $(\cP,\beta)$, where 
	\[
	\cP \text{ is a } G^\ad\text{-torsor on } \bA^1,\ \beta\colon 
	\cP_{\bA^1\setminus 0}\to G\times(\bA^1\setminus 0) 
	\text{ is a trivialization on } \bA^1\setminus 0.
	\] 
	Replacing the point 0 with an arbitrary $c\in\bA^1$, one gets a version $\Gr_c$ of the affine Grassmannian.
	Clearly, the isomorphism $\cO\simeq \cO_c=\bC[\![t-c]\!]$ induces the isomorphism $\Gr_c\simeq \Gr$ for any $c$.
	The schemes $\Gr_c$ glue together to the (trivial) bundle $\Gr_{\bA^1}$ over the affine line.
	\begin{rem}
		{\em $\Gr_{\bA^1}$ is the simplest example of a Beilinson-Drinfeld Grassmannian, the general case is discussed in the next section.} 
	\end{rem}

	\section{Beilinson-Drinfeld Schubert varieties}
	\label{bdg}
	We will need several versions of the Beilin\-son-Drinfeld Schubert varieties (see \cite{Z1,Z2}).
	Let us stress from the very beginning that the Beilinson-Drinfeld Grassmannians are defined over a (power of) a smooth curve $X$,
	but in this paper, we only consider the case $X=\bA^1$.
	The standard Beilinson-Drinfeld definition produces schemes over affine spaces.
	We will also need the symmetrized versions with the natural projections to the spectrum of the highest weight
	algebras. So we first discuss the properties of the highest weight algebras and then introduce the symmetrized 
	Beilinson-Drinfeld Schubert varieties.
	
	\subsection{The highest weight algebras}
	\label{high wei alg}
	Let $\ula=(\la_1,\dots,\la_k)$ be a multiset of dominant coweights.
	Let $\la=\sum_{i=1}^k \la_i=\sum_{j=1}^r m_j\om_j$ and  $N=\sum_{j=1}^r m_j=|\la|$.
	We set $$S_\lambda=S_{m_1}\times\ldots\times S_{m_r}.$$
	Recall the algebras $\cA(\la_1,\dots,\la_k)$ and 
	$$\cA_\la \simeq \cA(\underbrace{\om_1, \hdots ,
		\om_1}_{m_1}, \hdots , \underbrace{\om_r, \hdots , \om_r}_{m_r}) \simeq \bC[z_1,\dots,z_N]^{S_\la}$$
	of Section \ref{curr alg mod}.
	
	\begin{lem}
		There exists a natural surjection of algebras $\cA_\la \twoheadrightarrow \cA(\ula)$.
	\end{lem}
	\begin{proof}
		Note that $\cA_\la=\cA(\om_{a_1},\dots, \om_{a_N})$, where $\sum_{i=1}^n \om_{a_i}=\la$. Now it suffices to note that
		there exists a natural surjection
		\[
		\cA(\mu, \la_1,\dots,\la_k)\to \cA(\mu+\la_1,\dots,\la_k)
		\]
		induced by the surjection of the larger polynomial algebras
		\[
		\bC[z_1,\dots,z_{k+1}]\to \bC[z_1,\dots,z_k],\ z_1\mapsto z_1,\ z_2\mapsto z_1,\ z_i\mapsto z_{i-1},\ i>2. 
		\]
	\end{proof}

	Let $\BA^\lambda=\BA^N/S_\lambda=\Spec\CA_\lambda$ be the space of configurations of colored points on
	the line $\BA^1$ ($m_i$ points of color $\omega_i$). We have the main diagonal
	$\BA^\lambda\supset\BA^{(\lambda)}\simeq\BA^1$
	formed by all the configurations where all the points coincide. We have a finite morphism of addition
	of configurations $$\add\colon\BA^\nu\times\BA^\mu\to\BA^{\nu+\mu}.$$ Iterating it we obtain
	$$\add\colon\BA^{\lambda_1}\times\ldots\times\BA^{\lambda_k}\to\BA^\lambda.$$
	We define a closed subscheme $\BA^{(\ula)}\subset\BA^\lambda$ as the $\add$-image 
	of the closed subscheme
	$\BA^{(\lambda_1)}\times\ldots\times\BA^{(\lambda_k)}\subset\BA^{\lambda_1}\times\ldots\times\BA^{\lambda_k}$:
	\[
	\BA^{(\ula)}=
	\add\big(\BA^{(\lambda_1)}\times\ldots\times\BA^{(\lambda_k)}\big)\subset \BA^\lambda.
	\]

	\begin{lem}
		One has
		\[
		\bC[\bA^{(\ula)}]=\cA(\ula).
		\]
	\end{lem}
	
	\begin{proof}
		We denote the coordinates on $\BA^N$ by $x_{i,j}$, where $i=1,\dots,r$ and $j=1,\dots,m_i$.
		The group $S_\lambda$ acts by permuting the second indices.
		Let $\la_a=\sum_{b=1}^r m_{a,b} \om_b$ for $a=1,\dots,k$. In particular, $\sum_{a=1}^k m_{a,b}=m_b$
		for all $b=1,\dots,r$. Thus the coordinates with a fixed first index are divided into $k$ groups.
		For each $a=1,\ldots,k$, we combine the corresponding groups for all the possible first indices
		into one big group $\Gamma_a$. Now all the coordinates in $\BA^N$ are divided into groups
		$\Gamma_a,\ 1\leq a\leq k$.
		We consider the linear subspace $V$ in $\BA^N$ given by equations $x_{i,j}=x_{i',j'}$ whenever
		$(i,j)$ and $(i',j')$ lie in the same group $\Gamma_a$. We consider the saturation $S_\lambda V$
		(a union of a few vector subspaces in $\BA^N$). Finally, $\BA^{(\ula)}=(S_\lambda V)/S_\lambda$.
		Now the same argument as in the proof of~\cite[Proposition 2.2]{BCES} finishes our proof.
	\end{proof}
	
	\begin{rem}\label{abuse}
		{\em By construction, $\bA^k=\BA^{(\lambda_1)}\times\ldots\times\BA^{(\lambda_k)}$ is finite over
			$\bA^{(\ula)}$, cf.~\cite{DF, BCES}. For a closed point $\bc = (c_1, \hdots, c_k) \in \bA^k$ we sometimes
			keep the same notation for its image in $\BA^{(\ula)}$. For instance, by $\bC_\bc$ we usually mean
			the one-dimensional $\bC[\bA^{(\ula)}]$-module corresponding to the point $\bc$.}
		
	\end{rem}

	\subsection{BD Grassmannians and spherical Schubert varieties}
	The  Bei\-lin\-son-Drinfeld Grassmannian $\Gr_{\bA^k}$ ($BD$ Grassmannian for short) is the moduli space
	of collections consisting of 
	the points $(c_1,\dots,c_k)\in \bA^k$, a $G^\ad$-torsor $\cP$ over $\bA^1$, and a trivialization of
	$\cP$ outside the points $c_i$.
	\begin{example}
		{\em If $k=1$ then $\Gr_{\bA^1}$ is fibered over the affine line with a fiber isomorphic to the affine Grassmannian $\Gr$.} 
	\end{example}
	In general, the fiber of the natural projection $\pi\colon \Gr_{\bA^k}\to \bA^k$ over a point
	$(c_1,\dots,c_k)$ is isomorphic to the product of $a$ copies of $\Gr$, where $a$ is the number of
	distinct entries $c_i$. 
	\begin{example}
		{\em Let $\oA^k,\ k\ge 2$ be the open subvariety of $\BA^k$ consisting of points with pairwise
			distinct coordinates. Then} 
		\begin{equation}\label{openpart}
		\pi^{-1}(\oA^k)\simeq\oA^k\times \Gr^k.
		\end{equation} 
	\end{example}
	
	The BD Grassmannian $\Gr_{\BA^k}$ enjoys the key {\em factorization property}.
	We have the addition of configurations morphism $\add\colon\BA^k\times\BA^l\to\BA^{k+l}$ and
	an open subset $(\BA^k\times\BA^l)_{\on{disj}}\subset\BA^k\times\BA^l$ formed by all the pairs
	of disjoint effective divisors. Then there is a canonical isomorphism
	\[(\Gr_{\BA^k}\times\Gr_{\BA^l})|_{(\BA^k\times\BA^l)_{\on{disj}}}\cong
	\Gr_{\BA^{k+l}}\times_{\BA^{k+l}}(\BA^k\times\BA^l)_{\on{disj}}.\]
	
	The BD Grassmannian $\Gr_{\bA^k}$ is an ind-scheme, i.e.\ it is an inductive limit of the finite-dimensional 
	BD Schubert varieties $\ol\Gr{}^{\ula}$ for $k$-tuples of dominant coweights $\ula=(\la_1,\dots,\la_k)$.
	More precisely, we consider a group scheme $\BG(k)$ over $\bA^k$, whose fiber over a point
	$\bc=(c_1,\dots,c_k)\in\bA^k$
	is equal to the inverse limit ($n\to\infty$)
	\begin{equation}\label{inverselimit}
	\uG(k)_\bc=\varprojlim_{n} G^\sic(\bC[t]/P(t)^n),\ P(t)=\prod_{i=1}^k (t-c_i).
	\end{equation}
	Clearly, a fiber $\uG(k)_\bc$ is isomorphic to the $a$-th power of the group $G^\sic(\cO)$, where 
	$a$ is the number of distinct elements among $c_i$. 
	The group $\uG(k)$ naturally acts on $\Gr_{\bA^k}$ fiberwise.
	
	The spherical Schubert varieties in the BD Grassmannian are the closures of the
	$\uG(k)$-orbits in the fiber over the generic point of $\bA^k$. The orbits are parametrized by the
	$k$-tuples $\ula\in P_+^k$. Given such a collection, let 
	$t^\ula\colon\bA^k\to \Gr_{\bA^k}$ be a section of $\pi$ such that for $\bc\in\oA^k$ one has
	\[
	t^\ula(\bc)=((t-c_1)^{\la_1},\ldots,(t-c_k)^{\lambda_k})\in \prod_{i=1}^k \Gr_{c_i} 
	\]
	(so the total section is the closure of $t^\ula(\oA^k)$). 
	Now the BD Schubert varieties are defined as the closures of the $\uG(k)$-orbits: 
	\[
	\ol\Gr{}^{\ula}=\ol{\uG(k).t^\ula}\subset \Gr_{\bA^k}.
	\]    
	
	The restriction of $\pi\colon\Gr_{\BA^k}\to\BA^k$ to $\ol\Gr{}^{\ula}$ is denoted by
	$\pi_\ula\colon \ol\Gr{}^{\ula}\to \bA^k$. This is a flat morphism, and all the fibers are
	reduced~\cite[Proposition 1.2.4]{Z1} (it is proved for $k=2$ in {\em loc.\ cit.}, but the
	proof works for arbitrary $k$).
	The fiber $\ol\Gr{}^{\ula}_\bc=\pi_\ula^{-1}(\bc)$ over a point
	$\bc\in \bA^k$ with
	\begin{equation}\label{point}
	c_1=\ldots =c_{i_1}\ne c_{i_1+1}=\ldots = c_{i_1+i_2}\ne \ldots \ne c_{k-i_s+1}=\ldots =c_k
	\end{equation}
	is isomorphic to the product 
	\[
	\ol\Gr{}^{\la_1+\ldots+\la_{i_1}}\times\ol\Gr{}^{\la_{i_1+1}+\ldots+\la_{i_1+i_2}}\times\ldots
	\times\ol\Gr{}^{\la_{k-i_s+1}+\ldots+\la_k} 
	\]
	of spherical  Schubert varieties in the affine Grassmannian $\Gr$.
	In particular, the fiber of $\pi_\ula$ over
	the origin (or any other point of the total diagonal) is isomorphic to the spherical
	Schubert variety $\ol\Gr{}^{\la_1+\ldots+\la_k}$. 
	
	The BD Grassmannians and the BD Schubert varieties carry the relatively very ample determinant
	line bundle $\cL$. In particular, for any $\ell\ge 1$
	\[
	H^0(\ol\Gr{}^\ula_\bc,\cL^{\T\ell}_\bc)^* \simeq D_{\ell,\la_1+\ldots+\la_{i_1}}\T\dots \T
	D_{\ell,\la_{k-i_s+1}+\ldots+\la_k},
	\] 
	where $\cL_\bc$ is the restriction of the line bundle $\cL$ to the fiber $\ol\Gr{}^\ula_\bc$.
	
	\bigskip
	
	We introduce also the partially symmetrized (colored) version $\ol\Gr{}^{(\ula)}$ of the
	BD Schubert varieties. To define it we first consider the case of fundamental coweights $\la_i$. 
	So assume that all $\la_i$ are fundamental, i.e.
	\[
	\la_1=\dots=\la_{m_1}=\om_1,\dots, \la_{k-m_r+1}=\dots=\la_k=\om_r.
	\] 
	Let $\la=\sum_{i=1}^k \la_i=\sum_{j=1}^r m_j\om_j$, and $N=|\lambda|=k$.
	The action of $S_\lambda=S_{m_1}\times\ldots\times S_{m_r}$ on $\BA^N$ lifts to an action of
	$S_\lambda$ on $\Gr_{\BA^N}$ such that $\pi_{\ula}$ is $S_\lambda$-equivariant. We define
	$$\Gr_{\BA^\lambda}:=\Gr_{\BA^N}/S_\lambda.$$ 
	It is the moduli space of $G^\ad$-torsors on $\BA^1$
	trivialized away from an $N$-tuple of points $(x_1,\ldots,x_N)$, but we disregard the order within
	the groups $(x_1,\ldots,x_{m_1}),\ldots,(x_{N-m_r+1},\ldots,x_N)$.
	We introduce the closed subvariety $\ol\Gr{}^{(\ula)}\subset \Gr_{\BA^\lambda}$ as the categorical
	quotient
	\[
	\ol\Gr{}^{(\ula)}=\ol\Gr{}^{\ula}/S_\lambda\subset\Gr_{\BA^N}/S_\lambda=\Gr_{\BA^\lambda}.
	\]
	Since the collection $\ula$ of fundamental weights is uniquely
	determined by their sum $\lambda$, we also use the notation $\ol\Gr_\lambda$ for $\ol\Gr{}^{(\ula)}$.
	
	Now we consider an arbitrary $k$-tuple $\ula=(\lambda_1,\ldots,\lambda_k)$ (so that $\lambda_i$
	are not necessarily fundamental coweights). We set again $\lambda=\lambda_1+\ldots+\lambda_k$.
	Recall the closed subscheme $\BA^{(\ula)}\subset\BA^\lambda$ introduced in Section \ref{high wei alg}.
	We set 
	$$\ol\Gr{}^{(\ula)}:=\ol\Gr_\lambda\times_{\BA^\lambda}\BA^{(\ula)}.$$
	The natural projection
	$\ol\Gr{}^{(\ula)}\to\BA^{(\ula)}$ is denoted $\pi_{(\ula)}$.
	
	Note that in case when $(\ula) = (\la, \la, \hdots, \la)$, $\Gr^{(\ula)}$ is a Schubert variety in the symmetrized version of Beilinson-Drinfeld Grassmannian (see \cite{Z2}).
	
	The determinant line bundle $\cL$ descends from $\Gr_{\BA^N}$ to $\Gr_{\BA^\lambda}$.
	We will keep the same notation $\cL$ for its restriction to $\ol\Gr_\lambda$ and to $\ol\Gr{}^{(\ula)}$.

	\begin{prop} \label{sections under symmetrisation}
		\textup{(a)} Let $\la_1,\dots,\la_k$ be fundamental coweights, $\la=\sum_{i=1}^k \la_i$.
		Then one has the base change isomorphism:
		\begin{gather*}
		H^0(\ol\Gr{}^{\ula},\cL^{\T\ell})\cong H^0(\ol\Gr_\lambda,\cL^{\T\ell})\T_{\cA_\la} \BC[\bA^k].
		\end{gather*}
		\textup{(b)}  Let $\la_1,\dots,\la_k$ be arbitrary dominant coweights, $\la=\sum_{i=1}^k \la_i$.
		Then one has the base change isomorphism:
		\begin{gather*}
		  H^0(\ol\Gr{}^\ula,\cL^{\T\ell})\cong H^0(\ol\Gr{}^{(\ula)},\cL^{\T\ell})\T_{\cA(\ula)} \bC[\bA^k].
                \end{gather*}
\textup{(c)} The $\bC[\bA^k]$-module $H^0(\ol\Gr{}^\ula,\cL^{\T\ell})$ is free.

\noindent \textup{(d)} The $\cA(\ula)$-module $H^0(\ol\Gr{}^{(\ula)},\cL^{\T\ell})$ is free.
		\end{prop}
	\begin{proof}
		We have a cartesian square
		\[\begin{CD}
		\ol\Gr{}^\ula @>>> \ol\Gr{}^{(\ula)}\\
		@VVV @VVV\\
		\BA^{(\lambda_1)}\times\ldots\times\BA^{(\lambda_k)} @>>> \BA^{(\ula)},
		\end{CD}\]
		and the determinant line bundle $\cL$ on $\ol\Gr{}^\ula$ is the pullback of the determinant
		line bundle $\cL$ on $\ol\Gr{}^{(\ula)}$. The pushforward of the relatively very ample
		line bundle $\CL^{\otimes\ell}$ from $\ol\Gr{}^{\ula}$ to $\BA^k$ is a locally free sheaf $\CV$.	
                Indeed, we already know that $\pi_{\ula}\colon\ol\Gr{}^{\ula}\to\BA^k$ is flat and all the
                fibers are reduced. But the dimension of the space of sections of $\cL^{\T\ell}$ restricted
                to any fiber is
                independent of the choice of fiber by~\cite[Theorem 1]{FL1} or~\cite[Theorem 1.2.2]{Z1}.

                Furthermore, the pushforward of $\CL^{\otimes\ell}$ from $\ol\Gr_\lambda$ to $\BA^\lambda$ is a direct
		summand of $S_\lambda$-invariants in the pushforward of $\CV$ from $\BA^k$ to $\BA^\lambda$.
		Hence the pushforward of $\cL^{\otimes\ell}$ from $\ol\Gr_\lambda$ to $\BA^\lambda$
		is a locally free sheaf $\CW$
		as well. Finally, $\pi_{(\ula)*}\cL^{\otimes\ell}$ is the restriction of $\CW$ to
		$\BA^{(\ula)}\subset\BA^\lambda$,
		and hence $\pi_{(\ula)*}\cL^{\otimes\ell}$ is a locally free sheaf on $\BA^{(\ula)}$ as well.
		In particular, it is flat, and it remains to apply the base change for the above cartesian square.  
                This proves (a) and (b).

                To prove (c) and (d) note that $H^0(\ol\Gr{}^{(\ula)},\cL^{\T\ell})$ is projective over $\cA(\ula)$
                because its fibers have the same dimension at every closed point. Since both $\cA(\ula)$ and
                $H^0(\ol\Gr{}^{(\ula)},\cL^{\T\ell})$ are non-negatively graded, and the degree~$0$ part of
                $\cA(\ula)$ is $\bC$, we conclude by the graded Nakayama lemma that 
                $H^0(\ol\Gr{}^{(\ula)},\cL^{\T\ell})$ is free over
                $\cA(\ula)$.\footnote{This last observation is due to Roman Travkin.}
        \end{proof}

	\section{Global modules}\label{Global modules}
	In this section, we prove several statements on the global modules defined in \cite{DF}. Although
	we will be mainly interested in the global Demazure modules, we start in a more general setup.
	
	
	
	So let $\{M_i\}_{i=1}^k$  be cyclic graded $\fg[t]$-modules with cyclic vectors of dominant nonzero weights $\{\iota(\la_i)\}_{i=1}^k$, such that $t\h[t]$ annihilates these cyclic vectors. 
	Recall that it was proved in \cite{DF} that for $\bc = (c_1, \hdots, c_k)$ lying in some Zariski-open subset of $\bC^k$ one has
	\begin{equation} \label{fiber of global module at generic point}
	R(M_1, \hdots , M_k) \T_{\cA(\ula)} \bC_\bc \simeq \bigotimes_{i = 1}^k M_i(c_i)
	\end{equation}
	(see Remark \ref{abuse}).
	
	It was also shown that the fiber of $R(M_1, \hdots , M_k)$ at 0 surjects to the fusion product:
	\begin{equation} \label{fiber surjects to fusion}
	R(M_1, \hdots , M_k) \T_{\cA(\ula)} \bC_0 \twoheadrightarrow M_1 \ast \hdots \ast M_k.
	\end{equation}
	In particular, this surjection is an isomorphism if and only if
	the equality of dimensions
	$\dim \big( R(M_1, \hdots , M_k) \T_{\cA(\ula)} \bC_0 \big) = \prod_{i = 1}^k \dim M_i$ holds.

	The next proposition shows that if the above surjection is an isomorphism,
        then the Zariski-open subset for which \eqref{fiber of global module at generic point} holds can be described explicitly.
	
	\begin{prop}
		Suppose an isomorphism $R(M_1, \hdots , M_k) \T_{\cA(\ula)} \bC_0 \simeq M_1 \ast \hdots \ast M_k$ holds. Then an isomorphism 
		$R(M_1, \hdots , M_k) \T_{\cA(\ula)} \bC_\bc \simeq \bigotimes_{i = 1}^k M_i(c_i)$ holds for any $\bc$ with pairwise distinct coordinates 
		$c_i \neq c_j$ (not just for $\bc$ in some open subset).
	\end{prop}
	\begin{proof}
		As explained in \cite{DF}, an isomorphism 
		\[R(M_1, \hdots , M_k) \T_{\cA(\ula)} \bC_0 \simeq M_1 \ast \hdots \ast M_k\] implies by the
		semi-continuity theorem the equality
		\[\dim(R(M_1, \hdots , M_k) \T_{\cA(\ula)} \bC_\bc) = \prod_{i = 1}^{k} \dim M_i\] for any $\bc$.
		Hence, it suffices to construct a surjection \[R(M_1, \hdots , M_k) \T_{\cA(\ula)} \bC_\bc
		\twoheadrightarrow \bigotimes_{i = 1}^k M_i(c_i)\] for $\bc$ with $c_i \neq c_j$.
		For any $i$ clearly one has $M_i[t] \twoheadrightarrow M_i[t] \T_{\cA(\la_i)} \bC_{c_i} \simeq M_i(c_i)$ and hence 
		\[R(M_1, \hdots, M_k) = \bigodot_{i = 1}^k M_i[t] \twoheadrightarrow \bigodot_{i = 1}^k M_i(c_i) \simeq \bigotimes_{i = 1}^k M_i(c_i)
		\]
		(the last isomorphism is proved in \cite[Proposition 1.4]{FeLo}).
		To show that this surjection factors through $R(M_1, \hdots, M_k) \T \bC_\bc$ one needs to show that the relations 
		$ht^n - (\langle\iota(\la_1),h\rangle c_1^n + \hdots + \langle\iota(\la_k),h\rangle c_k^n)$ hold in the right-hand side module for any $h \in \h$, which is clearly true.
	\end{proof}
	

	For a global module $R(M_1, \hdots, M_k)$ we denote by $R(M_1, \hdots, M_k)^\vee$ its $\cA(\ula)$-dual, i.e. 
	\[
	R(M_1, \hdots, M_k)^\vee={\rm Hom}_{\cA(\ula)}(R(M_1, \hdots, M_k), \cA(\ula)).
	\]

	\begin{rem}
		{\em We note that $R(M_1, \hdots, M_k)^\vee$ carries a natural structure of $\fg[t]$-module.
			However, while $R(M_1, \hdots, M_k)$ is cyclic,
			$R(M_1, \hdots, M_k)^\vee$ does not have to be cyclic or cocyclic.
			The simplest example pops up for $\fg=\msl_2$, $k=2$ and $\la_1=\la_2=\om$. We note that 	
			if $R(M_1, \hdots, M_k)$ is free over $\cA(\ula)$, then the $q$-character of $R(M_1, \hdots, M_k)^\vee$ is computed as 
			$\ch_q\big(R(M_1, \hdots, M_k) \T_{\cA(\ula)} \bC_0\big)|_{q\to q^{-1}}\cdot\ch_q \cA(\ula) $.
			Hence in our special case one has
			\[
			\ch_q \bW_{2\om}^\vee = q^{-1}+(z^2+2+z^{-2})+q...,
			\] 
			showing that $\bW_{2\om}^\vee$ is neither cyclic nor cocylic.}
	\end{rem}

	
	\begin{prop} \label{tensor product over A}
	  Assume that the weights of the cyclic vectors of $\g[t]$-modules $M_i$ are
          nonzero. Then there is an isomorphism of $\g[t] - \cA(\ula)$-bimodules:
		\[
		R(M_1, \hdots, M_k)^{\odot\ell}_{\cA(\ula)} \simeq R(M_1^{\odot\ell}, \hdots, M_k^{\odot\ell})
		\]
		(notation of Section \ref{curr alg mod}).
	\end{prop}
	
	\begin{proof}
		Recall that the action of the highest weight algebra comes from the $\U(\h[t])$-action. In this proof, we consider global modules with different highest weight algebras. So, we use the notation $\T_{\U(\h[t])}$ instead of $\T_{\cA(\ula)}$, although formally there is no difference.
		
		We first consider an isomorphism 
		\begin{align*}
		M_i[t] \T_{\U(\h[t])} M_i[t] &\xrightarrow{\sim} (M_i \T M_i)[t], \\
		v_1t^{k_1} \T_{\U(\h[t])} v_2t^{k_2} &\mapsto (v_1 \T v_2) t^{k_1 + k_2}
		\end{align*}
		(one can easily check that it is bijective and $\g[t]$-equivariant).
		Then we extend it to an isomorphism
		\[
		\underbrace{\left(\otimes_{i = 1}^k M_i[t]\right)\otimes_{\U(\h[t])}\ldots
			\otimes_{\U(\h[t])}\left(\otimes_{i = 1}^k M_i[t]\right)}_\ell \xrightarrow{\sim}
		\bigotimes_{i = 1}^k M_i^{\T\ell}[t].
		\]
		Considering the $\g[t]$-envelopes of the tensor products of cyclic vectors in both sides, we
		obtain the desired isomorphism 
		\[R(M_1, \hdots, M_k)^{\odot\ell}_{\cA(\ula)} \simeq R(M_1^{\odot\ell}, \hdots, M_k^{\odot\ell}).
		\]
	\end{proof}
	
	\begin{cor} \label{algebra of duals R's}
		There is a natural structure of a graded $\cA(\ula)$-algebra on the space 
		\[
		\bigoplus_{\ell\ge 0} R(M_1^{\odot\ell}, \hdots, M_k^{\odot\ell})^\vee
		\]
		(we set $R(M_1^{\odot 0}, \hdots, M_k^{\odot 0})^\vee = \cA(\ula)$).
	\end{cor}
	
	\begin{proof}
		The multiplication structure is given by the dual of the map 
		\[
		R(M_1^{\odot\ell_1 +\ell_2}, \hdots, M_k^{\odot\ell_1 +\ell_2}) \hookrightarrow 
		R(M_1^{\odot\ell_1}, \hdots, M_k^{\odot\ell_1}) \T_{\cA(\ula)} R(M_1^{\odot\ell_2}, \hdots, M_k^{\odot\ell_2}).
		\]
	\end{proof}
	
	To make a link to the Beilinson-Drinfeld setup we consider the global modules with all $M$'s
	being Demazure modules of the same 
	level $\ell$ and nonzero highest weights. Let us recall the notation
	\begin{gather*}
	\bD(\ell, \ula) = R(D_{\ell, \la_1}, \hdots, D_{\ell, \la_k}),\\
	\bD_{\ell,\la} = \bD(\ell,  \underbrace{\om_1, \hdots , \om_1}_{m_1}, \hdots, \underbrace{\om_r, \hdots , \om_r}_{m_r}),
	\end{gather*}
	where $\la=\sum_{i=1}^k \la_i=\sum_{j=1}^r m_j\om_j$.
	Recall also that $D_{1, \la}^{\odot\ell} \simeq D_{\ell, \la}$. We get the following version
	of~Corollary~\ref{algebra of duals R's}:
	
	\begin{cor}
		There is a natural structure of a graded $\cA(\ula)$-algebra on the space 
		\[
		\bigoplus_{\ell\ge 0}  \bD(\ell, \ula)^\vee.
		\]
		In particular, if all $\la_i$ are fundamental, we have an $\cA_\la$-algebra
		\[
		\bD_{\la}^\vee = \bigoplus_{\ell\ge 0} \bD_{\ell,\la}^\vee.
		\]
	\end{cor}

	\begin{rem}
		{\em Note that by construction this algebra is generated by its first homogeneous component, or,
			in other words, one has a surjection from the symmetric algebra:
			\[
			\on{Sym}^\bullet_{\cA(\ula)} \bD(1, \ula)^\vee 
			\twoheadrightarrow \bigoplus_{\ell\ge 0}\bD(\ell, \ula)^\vee.
			\]
			This means that the $\BA^{(\ula)}$-scheme \[\Proj (\bigoplus_{\ell\ge 0} \bD(\ell, \ula)^\vee)\] 
			is a closed subscheme of the projective space \[
			\bP_{\cA(\ula)}(\bD(1, \ula)) = \Proj \Big( \on{Sym}^\bullet_{\cA(\ula)} \bD(1, \ula)^\vee) \Big).\]
		} 
	\end{rem}
	
	\begin{rem}
		{\em Assume that $\fg$ is simply laced, all $\la_i$ are fundamental, $\la=\sum_{i=1}^k \la_i$.
			Then $\bD(1, \ula)\simeq \bW_\la$ is the global
			Weyl module. It is proved in \cite{Kato2} that the projective spectrum of the algebra $\bigoplus_{\la\in P_+} \bW_\la^*$ is isomorphic to
			the (formal version of) semi-infinite flag variety
                        (see also \cite{Kato1,KNS,BF1,BF2,BF3,FiMi,FeMa1}). There are two important differences
			between the algebras  $\bigoplus_{\la\in P_+} \bW_\la^*$ and $\bigoplus_{\ell\ge 0}\bD(\ell, \ula)^\vee$. First, the sum
			in the first algebra runs over the dominant integral weights, while in the second case the summation is performed over the nonnegative
			integers. Second, the dual in the first algebra is taken with respect to the ground field, while in the second algebra one 
			considers the duals with respect to the highest weight algebra.}
	\end{rem}

	Now we will prove that an arbitrary global Demazure module $\bD(\ell, \ula)$ is free over $\cA(\ula)$. We start with the simply-laced case.
	
	\begin{lem} \label{simply-laced global Demazure is free}
		Let $\g$ be of simply-laced type, $\ula = (\la_1, \hdots, \la_k)$, and $\la = \la_1 + \hdots + \la_k$. Then
		\[
		\bD(\ell, \ula) \T_{\cA(\ula)} \bC_0 \simeq D_{\ell, \la}.
		\]
	\end{lem}
	
	\begin{proof}
		Note that this Lemma was already proved in \cite[Proposition 3.2]{DF} in the case of all $\la_i$ being fundamental coweights. Thus, fundamental Demazure modules satisfy the condition of Proposition \ref{fusion is associative}. Using the associativity (Proposition \ref{fusion is associative} \textup{(c)}), we obtain the Lemma for arbitrary coweights.
	\end{proof}
	
	We proceed to an arbitrary type.
	
	\begin{prop} \label{global demazure is projective}
		\textup{(a)} One has an isomorphism of $\fg[t]$-modules
		\[
		\bD(\ell, \ula) \T_{\cA(\ula)}\bC_0\simeq D_{\ell,\la_1+\dots+\la_k}.
		\]
		\textup{(b)} The global Demazure module $\bD(\ell, \ula)$ is free over $\cA(\ula)$.
	\end{prop}
	Our proof uses ideas of \cite[Theorem 8]{FL2}.
	\begin{proof}
		We reduce the general case to the case of $\msl_2$, which is simply laced and hence follows from Lemma \ref{simply-laced global Demazure is free}.
		
		As we know (recall \eqref{fiber surjects to fusion}), $\dim(\bD(\ell, \ula) \T_{\cA(\ula)} \bC_0) \geq \dim(D_{\ell, \la})$, so it suffices to
		construct a surjection
		\begin{equation} \label{surjection demazure to fiber}
		\bD(\ell, \ula) \T_{\cA(\ula)} \bC_0 \twoheadleftarrow D_{\ell, \la}.
		\end{equation}
		As it was shown in \cite{FL2,J}, the defining relations of $D_{\ell, \la}$ are
		\begin{align*}
		\fn^+[t].v &= 0, & h.v &= \ell \langle\iota(\la),h\rangle v, & t\h[t].v &= 0, &
		(f_\beta t^s)^{k_\beta + 1}.v &= 0.
		\end{align*}
		Here $h \in \h$, $s \in \bZ_{\geq 0}$, and $f_\beta$ is the Chevalley generator corresponding to
		a positive root $\beta\,^\svee$. Finally,
		$k_\beta=\ell \max\{ 0,\langle\Lambda^\vee_0 +\iota(\lambda),-\beta + s\frac{(\beta,\beta)}{2}K\rangle \}$.
		
		The first three relations obviously hold in the left-hand side of \eqref{surjection demazure to fiber}, so it remains to show that the last one also does. Consider the $\msl_2$-triple $\msl_2^\beta$ corresponding to $\beta\,^\svee$.
		
		It was shown in \cite[Lemma 7]{FL2} that there is an $\msl_2^\beta[t]$-submodule of $D_{1, \la_i}$,
		that is isomorphic to $\widehat{\msl_2^\beta}$ Demazure module $D_{\epsilon, m\omega^\svee}$,
		where $\epsilon = (\beta, \beta)/2$, $m=\langle\iota(\la),\beta\rangle/\epsilon$, and $\om^\svee$
		is the fundamental weight of $\msl_2^\beta$. 
		It follows that there is an $\msl_2^\beta[t]$-submodule of $D_{\ell, \la_i}$, that is isomorphic to $\widehat{\msl_2^\beta}$ Demazure module $D_{\ell \epsilon, m\omega^\svee}$. Denote it by $M(\ell, \iota(\la_i))$.
		
		This induces an embedding
		$M(\ell, \iota(\la_i))[t] \hk D_{\ell, \la_i}[t]$ and hence, denoting the 
		highest vector of $D_{\ell, \la_i}$ by $v_i$ we have
		\begin{multline*}
		R(M(\ell, \iota(\la_1)), \hdots, M(\ell, \iota(\la_k))) \simeq \U(\msl_2^\beta[t]).\T_{i = 1}^k v_i \hk \\
		\U(\g[t]).\T_{i = 1}^k v_i \simeq R(D_{\ell, \la_1}, \hdots D_{\ell, \la_k}) \simeq \bD(\ell, \ula).
		\end{multline*}
		Thereby, one has a map
		\begin{equation}\label{map}
		R( M(\ell, \iota(\la_1)), \hdots, M(\ell,\iota(\la_k))) \T_{\cA} \bC_0 \to \bD(\ell, \ula) \T_{\cA(\ula)} \bC_0,
		\end{equation}
		where $\cA$ is the highest weight algebra of the global $\msl_2^\beta[t]$ Demazure module 
		$R(M(\ell, \iota(\la_1)), \hdots, M(\ell, \iota(\la_k)))$ and the map
		\eqref{map} is induced by the natural inclusion $\cA\subset \cA(\ula)$.
		Now the left hand side of \eqref{map} is isomorphic to $M(\ell, \iota(\la))$,
		since we are in the simply laced case $\fg=\msl_2^\beta$. 
		The required relations $(f_\beta t^s)^{k_\beta + 1}.v = 0$, $s\ge 0$,
		hold in this module, and hence they hold in $\bD(\ell, \ula) \T_{\cA(\ula)} \bC_0$ as well.

		
		
		
		
The end of the proof repeats the one of~Proposition~\ref{sections under symmetrisation}:
                $\bD(\ell, \ula)$ is projective over $\cA(\ula)$ because its fibers have the same dimension at
		every closed point. Now since both $\bD(\ell, \ula)$ and $\cA(\ula)$ are non-negatively
		graded, and the degree~0 part of $\cA(\ula)$ is $\BC$, we conclude by the graded Nakayama lemma
		that $\bD(\ell, \ula)$ is free over $\cA(\ula)$.
	\end{proof}

	
	
	

	
	Now we describe one more relation between global modules that will be used later in the paper:

	\begin{prop} \label{fat cartan component}
		One has an isomorphism of $\fg[t]$-modules:
		\[
		\bD(\ell, \ula) \T_{\cA(\ula)} \bigotimes_{i = 1}^k \cA(\ell\la_i)  \simeq 
		\U(\fg[t]). \Bigl( \bigotimes_{i = 1}^k D_{\ell,\la_i}[t]_{\ell\iota(\la_i)}  \Bigr) \subset \bigotimes_{i = 1}^k D_{\ell, \la_i}[t],
		\]
		where $D_{\ell,\la_i}[t]_{\ell\iota(\la_i)}$
		denotes the highest weight part of $D_{\ell, \la_i}[t]$.
	\end{prop}
	
	Note that each algebra $\cA(\ell\la_i)$ is isomorphic to the algebra of polynomials in one variable.
	We write $\cA(\ell\la_i)$ 
	(as opposed to just $\bC[z_i]$) to point out that these algebras come as the highest weight spaces
	of the modules $D_{\ell,\la_i}[t]$. 
	Further in the paper in the Beilinson-Drinfeld context we use the notation
	$\bC[\bA^{(\ula)}] \subset \bC[\bA^k]$ instead of 
	$\cA(\ula) \subset \bigotimes_{i=1}^k \cA(\ell\la_i)$, although it is the same, as explained
	in~Subsection \ref{high wei alg}.
	
	\begin{proof}
		Let $v$ be the cyclic vector of $\bD(\ell, \ula)$ and let $v_i$ be the cyclic vector of $D_{\ell, \la_i}[t]$.
		We define the desired morphism by setting:
		\begin{align*}
		\phi\colon \bD(\ell, \ula) \T_{\cA(\ula)} \bigotimes_{i = 1}^k \cA(\ell\la_i) &\rightarrow \bigotimes_{i = 1}^k D_{\ell, \la_i}[t], \\
		u.v \T_{\cA(\ula)} (h_1t^{s_1} \T \hdots \T h_nt^{s_k}) &\mapsto u(h_1t^{s_1} v_1 \T \hdots \T h_nt^{s_k} v_k),
		\end{align*}
		for $u \in  \U(\g[t])$.
		
		It is well defined because $\cA(\ell\la_i)$ acts on $D_{\ell, \la_i}[t]$, and hence $\bigotimes_{i = 1}^k \cA(\ell\la_i)$ acts on 
		$\bigotimes_{i = 1}^k D_{\ell, \la_i}[t]$ commuting with the $\g[t]$-action.
		
		It remains to prove injectivity. Consider both sides as $\bigotimes_{i = 1}^k \cA(\ell\la_i)$-modules. 
		Fibers of both sides at any point $\bc = (c_1, \hdots, c_k)$ with $c_i \neq c_j$ are $\bigotimes_{i = 1}^k D_{\ell, \la_i}(c_i)$. 
		Thereby, $\phi$ is injective on fibers in an open subset, and hence injective.
	\end{proof}
	
	\begin{example}
		{\em Let $\fg=\msl_2$, $k=2$, $\la_1=\la_2=\om$. Then one has two embeddings:
			\begin{gather*}
			\bW_{2\om} \hookrightarrow \bW_\om \T \bW_\om,\\
			\bW_{2\om} \T_{\bC[z_1, z_2]^{S_2}} \bC[z_1, z_2] \hookrightarrow \bW_\om \T \bW_\om.
			\end{gather*}
			The image of the first embedding (the special case of Kato's theorem \cite[Corollary 3.5]{Kato1}) is the $\U(\g[t])$-envelope of the tensor product 
			of the highest vectors, while the image of the second embedding is the $\U(\g[t])$-envelope of the tensor product of the highest weight components.}
	\end{example}
	
	\begin{rem} \label{character of fat demazure}
		{\em Due to Proposition \ref{global demazure is projective}, $\bD(\ell, \ula)$ is free over $\cA(\ula)$. 
			Hence, $\bD(\ell, \ula) \T_{\cA(\ula)} \bigotimes_{i = 1}^k \cA(\ell\la_i)$ is free over
			$\bigotimes_{i = 1}^k \cA(\ell\la_i)$.  
			In particular, this implies:}
		\[
		\ch_q(\bD(\ell, \ula) \T_{\cA(\ula)} \bigotimes_{i = 1}^k \cA(\ell\la_i)) = \ch_q D_{\ell, \la_1 + \hdots + \la_k} \times (1 - q)^{-k}.
		\]
	\end{rem}
	
	\begin{rem}
		{\em In fact, the proof of Proposition \ref{fat cartan component} hold for arbitrary global modules $R(M_1, \hdots, M_n)$ without any changes. The isomorphism in the general case is of the form 
			\[
			R(M_1,\hdots ,M_k) \T_{\cA(\ula)} \bigotimes_{i = 1}^k \cA(\la_i)  \simeq \U(\fg[t]) . \Bigl( \bigotimes_{i = 1}^k \U(\h[t]) v_i  \Bigr) 
			\subset \bigotimes_{i = 1}^k R(M_i).
			\]}
	\end{rem}

	\section{Global Demazure modules and BD Schubert varieties}\label{GDm and BD}
	\subsection{Sections of the determinant line bundle}
	The goal of this section is to identify the global Demazure modules $\BD(\ell,\ula)$ with the $\cA(\ula)$-dual
	of the space of sections $H^0(\ol\Gr{}^{(\ula)},\cL^{\T\ell})$ (we note that the higher cohomology
	$H^{>0}(\ol\Gr{}^{(\ula)},\cL^{\T\ell})$ vanish: as in the proof
	of~Proposition~\ref{sections under symmetrisation}, it follows from the flatness of $\pi_{(\ula)}$
	and the fact that the restriction of $\CL^{\otimes\ell}$ to any fiber of $\pi_{(\ula)}$ is very
	ample).
	To this end, we first establish an isomorphism 
	\[
	H^0(\ol\Gr{}^{\ula},\cL^{\T\ell})^\vee\simeq \BD(\ell,\ula)\T_{\cA(\ula)} \bC[\bA^k],
	\]
	where $\BA^k=\BA^{(\lambda_1)}\times\ldots\times\BA^{(\lambda_k)}$ and the notation $M^\vee$ stands for
	the $\bC[\bA^k]$-dual module to a $\bC[\bA^k]$-module $M$.
	In order to compare these two spaces we make the following observation.
	
	\begin{lem}
		\label{above}
		There is a homomorphism of Lie groups $$G^\sic[t]\to \Gamma(\bA^k,\uG(k)),$$ where
		$\Gamma(\bA^k,\uG(k))$ is the group of sections of the group scheme $\uG(k)$ over $\BA^k$.
	\end{lem}
	\begin{proof}
		Recall that $\uG(k)$ is defined as a scheme over $\bA^k$ whose fiber over a point $\bc$ is equal to the 
		inverse limit ($n\to\infty$) of the groups $G^\sic(\bC[t]/P(t)^n)$, where $P(t)=\prod_{i=1}^k (t-c_i)$.
		Now the desired homomorphism is induced by sending the coordinate $t$ in
		$\BC[t]$ to $t\pmod{P(t)^n}$. 
	\end{proof}
	
	\begin{cor}
		The space of sections $H^0(\ol\Gr{}^{\ula},\cL^{\T\ell})$ is a $\fg[t]$-module. The $\fg[t]$-action
		commutes with the natural action of $\bC[\bA^k]$.
	\end{cor}
	\begin{proof}
		The first claim is a direct consequence of Lemma~\ref{above}. The second claim is clear since the group scheme $\uG(k)$ acts fiberwise.
	\end{proof}
	
	Now we prove our claim for the case of one weight.
	
	\begin{lem}\label{k=1}
		Let $k=1$, i.e.\ $\ula=(\la)$, $\la\in P_+$. Then for any $\ell\ge 1$ we have an isomorphism of $\fg[t]$-modules
		$$H^0(\ol\Gr{}^{(\la)},\cL^{\T\ell})^\vee\simeq D_{\ell,\la}[t].$$
	\end{lem}
	\begin{proof}
		Recall formula \eqref{formula} for the action of the Lie algebra $\fg[x]$ on  $D_{\ell,\la}[t]\simeq D_{\ell,\la}\T\bC[t]$: 
		\begin{equation}\label{g}
		(g\T x^s) (v\T t^a)= \sum_{i=0}^s (-1)^{s-i}\binom{s}{i} (g\T x^i. v)\T t^{a+s-i},\ g\in\fg.
		\end{equation}
		Here we deliberately replaced the variable $t$ in $\fg[t]$ with an auxiliary variable $x$ in order to make the picture similar to the BD context. 
		
		Now let us identify $x$ with the global coordinate on $\bA^1$. Then one gets an isomorphism of vector spaces
		\[
		H^0(\ol\Gr{}^{(\la)},\cL^{\T\ell})^\vee\simeq D_{\ell,\la}\T\bC[x],
		\]
		where $D_{\ell,\la}$ is considered as a $\fg[t]$-module.
		The action of $\fg[x]$ is induced by the map $x\mapsto x-t$, meaning that the result of the action of $g\T x^s$ on $v\T t^a$
		is given by the right hand side of  \eqref{g}. 
	\end{proof}

	\begin{thm} \label{sections of nonsymmetrised bd}
		Let $\ula=(\la_1,\dots,\la_k)\in P_+^k$, all $\la_i$ are nonzero. Then one has an isomorphism of $\fg[t]-\bC[\bA^k]$-bimodules:
		\[
		H^0(\ol\Gr{}^{\ula},\cL^{\T\ell})^\vee\simeq \BD(\ell,\ula)\T_{\cA(\ula)} \bC[\bA^k],
		\]
		where $M^\vee$ stands for the $\bC[\bA^k]$-dual module to a $\bC[\bA^k]$-module $M$.
		
	\end{thm}
	\begin{proof}
                According to~Proposition~\ref{sections under symmetrisation}(c),
                $H^0(\ol\Gr{}^{(\la)},\cL^{\T\ell})^\svee$ is free as a $\bC[\bA^k]$-module.
                In particular, its $\fg$-highest weight part $H^0(\ol\Gr{}^{\ula},\cL^{\T\ell})^\vee_{\ell\iota(\lambda)}$
                is isomorphic to the free rank one module over $\bC[\bA^k]$. We also conclude that
		\begin{equation} \label{ch of section}
		\ch_q H^0(\ol\Gr{}^{\ula},\cL^{\T\ell})^\vee = \ch_q D_{\ell,\lambda}\cdot(1-q)^{-k}.
		\end{equation}
		
		Due to Proposition \ref{fat cartan component} and Lemma \ref{k=1}, in order to prove the Theorem, it is enough to show:
		\begin{equation}\label{embed}
		H^0(\ol\Gr{}^{\ula},\cL^{\T\ell})^\vee\simeq \U(\fg[t]).
		\bigotimes_{i=1}^k H^0(\ol\Gr{}^{(\la_i)},\cL^{\T\ell})^\vee_{\ell\iota(\la_i)},
		\end{equation}
		where the lower index denotes the corresponding $\fg$-weight subspace.

		By~Lemma~\ref{k=1}, $H^0(\ol\Gr{}^{(\la_i)},\cL^{\T\ell})^\vee_{\ell\iota(\la_i)}$ is isomorphic to the
		polynomial ring in one variable as a vector space. We consider the embedding 
		\begin{equation}\label{nodiagonal}
		H^0(\ol\Gr{}^{\ula},\cL^{\T\ell})^\vee \hookrightarrow H^0(\ol\Gr{}^{\ula}(\oA^k),\cL^{\T\ell})^\vee, 
		\end{equation}
		where $\ol\Gr{}^{\ula}(\oA^k)\subset \ol\Gr{}^{\ula}$ is $\pi_{\ula}^{-1}(\oA^k)$, and the
		embedding~\eqref{nodiagonal} is induced by the open embedding 
		$\ol\Gr{}^{\ula}(\oA^k)\hookrightarrow \ol\Gr{}^{\ula}$. 
		By the factorization property,
		\begin{equation}\label{fp}
		H^0(\ol\Gr{}^{\ula}(\oA^k),\cL^{\T\ell})^\vee\cong \bC[\oA^k] \otimes \bigotimes_{i=1}^k D(\ell,\la_i).
		\end{equation}
		In particular, the highest weight part $H^0(\ol\Gr{}^{\ula}(\oA^k),\cL^{\T\ell})_{\ell\iota(\la)}^\vee$
		(with $\la=\sum_{i=1}^k \la_i$)
		is a free rank one module over the localization $\bC[\oA^k]$ of the polynomial algebra $\bC[\bA^k]$.
		
		Thus we have the embeddings
		\begin{multline*}
		H^0(\ol\Gr{}^{\ula},\cL^{\T\ell})^\vee_{\ell\iota(\lambda)}\hookrightarrow
		H^0(\ol\Gr{}^{\ula}(\oA^k),\cL^{\T\ell})^\vee_{\ell\iota(\la)}\cong
		\bC[\oA^k] \otimes \bigotimes_{i=1}^k D(\ell,\la_i)_{\ell\iota(\lambda_i)}\\
		\hookleftarrow\bC[\bA^k]\otimes_{\cA(\ula)}\bD(\ell,\ula)_{\ell\iota(\lambda)}
		\end{multline*}
		arising from the factorization property.
		
		We claim that the images of these embeddings coincide. First we consider the case $k=2$.
		Then the image $I_{\on{left}}$ of the left embedding and the image $I_{\on{right}}$
		of the right embedding both are the free rank one
		modules over $\bC[\bA^2]$ inside the free rank one module over $\bC[\oA^2]$.
		If we denote the coordinates in $\BA^2$ by $z_1,z_2$, then necessarily
		$I_{\on{left}}=(z_1-z_2)^aI_{\on{right}}$ for some $a\in\BZ$, and we have to prove $a=0$.
		Otherwise either $I_{\on{left}}\subsetneq I_{\on{right}}$ (if $a>0$) or
		$I_{\on{right}}\subsetneq I_{\on{left}}$ (if $a<0$). In the first case the graded character of
		$I_{\on{left}}$ is {\em strictly less} than the graded character of $I_{\on{right}}$ (termwise)
		that contradicts the equality $\ch_q H^0(\ol\Gr{}^{\ula},\cL^{\T\ell})^\vee
		=\ch_q D_{\ell,\lambda}\cdot(1-q)^{-k}=\ch_q\big(\bC[\bA^k]\otimes_{\cA(\ula)}\bD(\ell,\ula)\big)$
		by \eqref{ch of section} and~Remark~\ref{character of fat demazure}
		(in particular, the graded characters of the $\ell\iota(\lambda)$-weight components must coincide
		as well). The second case similarly leads to a contradiction.
		
		The coincidence of images for general $k$ now follows after localization at generic points of
		diagonals in $\BA^k$ by factorization. Since we know the coincidence generically and in
		codimension one, it follows everywhere by algebraic Hartogs' lemma: given
		two locally free sheaves on $\bA^k$, an isomorphism between them defined off a codimension~2 closed subset of $\bA^k$ necessarily extends to the whole of $\bA^k$.
		
		We conclude the equality 
		\[
		H^0(\ol\Gr{}^{\ula},\cL^{\T\ell})^\vee_{\ell\iota(\lambda)}=
		\bC[\bA^k]\otimes_{\cA(\ula)}\bD(\ell,\ula)_{\ell\iota(\lambda)}\]
		inside
		$\bC[\oA^k] \otimes \bigotimes_{i=1}^k D(\ell,\la_i)$.
		But 
		\begin{multline*}
		H^0(\ol\Gr{}^{\ula},\cL^{\T\ell})^\vee\supset
		\U(\fg[t])H^0(\ol\Gr{}^{\ula},\cL^{\T\ell})^\vee_{\ell\iota(\lambda)}\\
		=\U(\fg[t])\big(\bC[\bA^k]\otimes_{\cA(\ula)}\bD(\ell,\ula)_{\ell\iota(\lambda)}\big)
		=\bC[\bA^k]\otimes_{\cA(\ula)}\bD(\ell,\ula).
		\end{multline*}
		The equality of characters $\ch_q H^0(\ol\Gr{}^{\ula},\cL^{\T\ell})^\vee=
		\ch_q\big(\bC[\bA^k]\otimes_{\cA(\ula)}\bD(\ell,\ula)\big)$ once again guarantees that the above
		inclusion is actually an equality.
		
		The theorem is proved. 
	\end{proof}
	
	
	Now to establish a relation between the global Demazure modules and the spaces of sections of determinant line bundles on the symmetrized Schubert varieties, we prove the following theorem. 
	
	\begin{thm}\label{main theorem}
		Let $\la_1,\dots,\la_k\in P_+$, all $\la_i$ are nonzero. Then one has an isomorphism of $\fg[t]$-modules:
		\[
		H^0(\ol\Gr{}^{(\ula)},\cL^{\T\ell})^\vee\simeq \bD(\ell,\ula),
		\]
		where $M^\vee$ stands for the $\cA(\ula)$-dual module to an $\cA(\ula)$-module $M$.
		
	\end{thm}
	\begin{proof}
		Using Proposition \ref{sections under symmetrisation} and Theorem \ref{sections of nonsymmetrised bd} we get an isomorphism:
		\begin{equation} \label{sections times polynomials}
		\bD(\ell,\ula)\T _{\cA(\ula)} \bC[\bA^k] \simeq H^0(\ol\Gr{}^{(\ula)},\cL^{\T\ell})^\vee\T_{\cA(\ula)} \bC[\bA^k].
		\end{equation}
		Let $v$ be the cyclic vector of $\bD(\ell, \ula)$. The vector $v \T_{\cA(\ula)} 1$ in the left hand
		side of \eqref{sections times polynomials} is mapped to some vector of the form $w \T_{\cA(\ula)} 1$
		in the right hand side. Using the $\g[t]$-equivariance we obtain that 
		\[
		\bD(\ell, \ula) \simeq (\U(\g[t]).w) \T_{\cA(\ula)} 1 \subset H^0(\ol\Gr{}^{(\ula)},\cL^{\T\ell})^\vee \T_{\cA(\ula)} 1.
		\]
		Hence there is an embedding $\bD(\ell, \ula) \hookrightarrow H^0(\ol\Gr{}^{(\ula)},\cL^{\T\ell})^\vee$.
		Using Proposition~\ref{global demazure is projective}, we see that the fibers at $0$ of both
		sides are isomorphic to $D_{\ell, \la_1 + \hdots +\la_k}$. The graded version of Nakayama lemma implies
		that the above injective map is a surjection, and thus an isomorphism.
	\end{proof}
	
	\begin{rem}
		If the highest weight of a cyclic $\g[t]$ module $M$ is zero, then the module $M[z]$ is not cyclic. That is why we impose the condition $\la_i \neq 0$ in Theorems \ref{sections of nonsymmetrised bd} and \ref{main theorem}.
		
		It is an easy consequence of these Theorems that for the case $\ula = (\umu, \underbrace{0, \hdots, 0}_{n})$ with $\mu_i \neq 0$ one has
		\begin{align*}
		H^0(\ol\Gr{}^{\ula}, \cL^{\T\ell})^\vee &\simeq \big( \BD(\ell,\umu) \T_{\cA(\umu)} \bC[\bA^k] \big) \T \bC[\bA^n]; \\
		H^0(\ol\Gr{}^{(\ula)},\cL^{\T\ell})^\vee &\simeq \bD(\ell,\umu) \T \bC[\bA^{(n)}].
		\end{align*}
	\end{rem}
	
	\begin{cor}
		One has an isomorphism of $\bA^{(\ula)}$-schemes:
		\[
		\ol\Gr{}^{(\ula)}\simeq \Proj\left(\bigoplus_{\ell\ge 0}\BD(\ell,\ula)^\vee\right). 
		\]
	\end{cor}

	Let us consider the special case when all $\la_i$ are fundamental coweights. In particular,
	$\cA(\ula)\simeq \cA_\la$. We obtain the following corollary.
	\begin{cor}
		Assume that all $\la_i$ are fundamental coweights and let $\la=\sum_{i=1}^k \la_i$. Then
		\begin{enumerate}
			\item $H^0(\ol\Gr{}^{(\ula)},\cL^{\T\ell})^\vee\simeq \bD(\ell,\la)$;
			\item $\ol\Gr{}^{(\ula)}\simeq \Proj\left(\bigoplus_{\ell\ge 0} \bD(\ell,\la)^\vee\right)$.
		\end{enumerate}
	\end{cor}

	\subsection{Embeddings of the BD Schubert varieties}
	The goal of this section is to show that the global Demazure modules provide projective embeddings
	of  Beilinson-Drinfeld 
	Schubert varieties (generalizing a relation between the affine Demazure modules and Schubert varieties).   
	
	Thanks to Proposition \ref{global demazure is projective} the global Demazure module $\bD(\ell,\ula)$ is free over $\cA(\ula)$.
	Hence one gets a vector bundle $\cD(\ell,\ula)$ on 
	$\bA^{(\ula)}=\Spec(\cA(\ula))$, whose fiber is given by the fiber of $\bD(\ell,\ula)$ at a point of the base.
	We will need the following lemma in order to embed  the BD Schubert varieties into the fiberwise projectivized vector bundle
	$\cD(1,\ula)$.
	
	\begin{lem}\label{group scheme}
		The group scheme $\uG(k)$ acts on $\cD(\ell,\ula)$ fiberwise.
	\end{lem} 
	\begin{proof}
		Recall (see \eqref{inverselimit}) that the fiber of $\BG(k)$  over a point $\bc=(c_1,\dots,c_k)\in\bA^k$
		is equal to the inverse limit 
		\[
		\uG(k)_\bc=\varprojlim_{m} G^\sic(\bC[t]/P(t)^m),\ P(t)=\prod_{i=1}^k (t-c_i).
		\]
		We also know that for $\bc = (\underbrace{c_1, \hdots, c_1}_{i_1}, \hdots, \underbrace{c_n, \hdots, c_n}_{i_n}) \in \bC^k$
		such that $c_p \neq c_q$ for $p \neq q$ one has
		\[
		\bD(\ell, \ula) \T_{\cA(\ula)} \bC_\bc \simeq \bigotimes_{p = 1}^n D_{\ell, \la_{i_1+\dots+i_{p-1}}+\dots+\la_{i_1+\dots+i_p}}(c_{i_p}).
		\]
		We conclude that $\uG(k)$ acts on $\cD(\ell,\ula)$ fiberwise.
	\end{proof}
	
	One has a section
	$s^\ula\colon \bA^{(\ula)}\to \bP(\cD(1,\ula))$ of the natural projection map
	$\bP(\cD(1,\ula))\to  \bA^{(\ula)}$ sending a point $\bc$ to the highest weight line of $\bD(1, \ula) \T_{\cA(\ula)} \bC_\bc$.
	By Lemma \ref{group scheme} the group of sections of the group scheme $\uG(k)$ naturally acts on $\bP(\cD(1,\ula))$.
	We obtain the following corollary.
	\begin{cor}
		$\ol\Gr{}^{(\ula)}$ is equal to the 
		closure of the $\uG(k)$-orbit of the section $s^\ula$.
	\end{cor}
	\begin{proof}
		Follows from definition of $\ol\Gr{}^{\ula}$ in Section \ref{bdg} and Theorem \ref{main theorem}.
	\end{proof}

	\appendix
	\section{On the associativity of fusion product}
	
	It was conjectured in \cite{DF} that
	\begin{equation}\label{Demazure-fusion}
	R(M_1, \hdots, M_k) \T_{\cA (\ula)} \bC_0 \simeq M_1 \ast \hdots \ast M_k.
	\end{equation}
	The existence of the isomorphism \eqref{Demazure-fusion} implies that fusion product does not depend on the choice of constants. 
	Now we prove that \eqref{Demazure-fusion} also implies the associativity of the fusion product.
	
	\begin{prop} \label{fusion is associative}
		Let $N_1, \hdots, N_k, M_1, \hdots, M_m$ be finite-dimensional graded cyclic $\fg[t]$-modules with cyclic vectors of weights 
		$\la_1^\svee, \hdots, \la_k^\svee, \mu_1^\svee, \hdots, \mu_m^\svee$ such that 
		\[
		R(N_1, \hdots, N_k, M_1, \hdots, M_m) \T_{\cA(\ula^\svee, \umu^\svee)} \bC_0 \simeq N_1 \ast \hdots \ast N_k \ast M_1 \ast \hdots \ast M_m.
		\]
		Then 
		
		\textup{(a)}
		\[
		R(N_1, \hdots, N_k) \T_{\cA(\ula^\svee)} \bC_0 \simeq N_1 \ast \hdots \ast N_k,
		\]
		
		\textup{(b)}
		\[
		N_1 \ast \hdots \ast N_k \ast M_1 \ast \hdots \ast M_m \simeq (N_1 \ast \hdots \ast N_k) \ast M_1 \ast \hdots \ast M_m,
		\]
		
		\textup{(c)}
		\begin{multline*}
		  R(R(N_1, \hdots, N_k) \T_{\cA(\ula^\svee)} \bC_0, M_1, \hdots, M_m )
                  \T_{\cA(\la_1^\svee + \cdots + \la_k^\svee, \umu^\svee)} \bC_0 \\
		\simeq R(N_1, \hdots, N_k, M_1, \hdots, M_m) \T_{\cA(\ula^\svee, \umu^\svee)} \bC_0.
		\end{multline*}
	\end{prop}
	
	\begin{proof}
		We consider pairwise distinct $c_0, c_1, \hdots,  c_m \in \bC$. Let 
		$$\bc = (\underbrace{c_0, \hdots, c_0}_k, c_1, \hdots, c_m) \in \bC^{k + m}.$$ 
		Then clearly
		\begin{multline} \label{fiber at partly equal coordinates}
		R(N_1, \hdots, N_k, M_1, \hdots, M_m) \T_{\cA(\ula^\svee, \umu^\svee)} \bC_\bc \twoheadrightarrow \\
		\Big( R(N_1, \hdots, N_k) \T_{\cA(\ula^\svee)} \bC_{(c_0, \hdots, c_0)} \Big) \odot \Big( R(M_1, \hdots, M_m) \T_{\cA(\umu^\svee)} \bC_{(c_1, \hdots, c_m)} \Big)
		\\ \simeq \Big( R(N_1, \hdots, N_k) \T_{\cA(\ula^\svee)} \bC_0 \Big) (c_0) \odot \Big( M_1(c_1) \T \hdots \T M_m(c_m) \Big) 
		\\ \simeq \big( R(N_1, \hdots, N_k) \T_{\cA(\ula^\svee)} \bC_0 \big) (c_0) \T (M_1(c_1) \T \hdots \T M_m(c_m)).
		\end{multline}
		The last isomorphism holds because of \cite[Proposition 1.4]{FeLo}.
		
		Because of our assumption on the fiber at 0 of $R(N_1, \hdots, N_k, M_1, \hdots, M_m)$ we conclude
		that the fibers at all the points have the same dimension, and
		surjection~\eqref{fiber at partly equal coordinates} implies
		\begin{align*}
		\prod_{i = 1}^k \dim N_i \times \prod_{i = 1}^m \dim M_i &\geq
		\dim \big( R(N_1, \hdots, N_k) \T_{\cA(\ula^\svee)} \bC_0 \big) \times \prod_{i = 1}^m \dim M_i; \\
		\prod_{i = 1}^k \dim N_i &\geq \dim \big( R(N_1, \hdots, N_k) \T_{\cA(\ula^\svee)} \bC_0 \big),
		\end{align*}
		Comparing with \eqref{fiber surjects to fusion}, we obtain part \textup{(a)} of Proposition. We also conclude that 
		\eqref{fiber at partly equal coordinates} is an isomorphism.
		
		Now, as proved in \cite[Proposition 2.11]{DF}, there is a surjection
		\begin{multline*}
		N_1 \ast \hdots \ast N_k \ast M_1 \ast \hdots \ast M_m \simeq
		R(N_1, \hdots, N_k, M_1, \hdots, M_m) \T_{\cA(\ula^\svee, \umu^\svee)} \bC_0 \twoheadrightarrow
		\\ \gr \big( R(N_1, \hdots, N_k, M_1, \hdots, M_m) \T_{\cA(\ula^\svee, \umu^\svee)} \bC_\bc \big)
		\\ \simeq \gr \Big( \big( R(N_1, \hdots, N_k) \T_{\cA(\ula^\svee)} \bC_0 \big) (c_0) \T (M_1(c_1) \T \hdots \T M_m(c_m)) \Big)
		\\ \simeq (N_1 \ast \hdots \ast N_k) \ast M_1 \ast \hdots \ast M_m.
		\end{multline*}
		Comparing dimensions of both sides, we obtain part \textup{(b)} of the Proposition.
		
		To prove the remaining part, we first prove that
		\begin{equation} \label{R with base changed to one variable}
		\big( R(N_1, \hdots, N_k) \T_{\cA(\ula^\svee)} \bC_0 \big) [z] \simeq R(N_1, \hdots, N_k) \T_{\cA(\ula^\svee)} \cA',
		\end{equation}
		where the algebra $\cA'$ is the algebra of polynomials in one variable, obtained by gluing all the variables in the algebra $\cA(\ula^\svee)$:
		\begin{center}
			\begin{tikzcd}
			\bC[z_1, \hdots, z_{k}] \arrow[r, "z_i \mapsto z"] & \bC[z] \\
			\cA(\ula^\svee) \arrow[u, hook] \arrow[r] & \cA' \arrow[u, equal].
			\end{tikzcd}
		\end{center}
		
		Indeed, it follows from part \textup{(a)} of the Proposition that $R(N_1, \hdots, N_k,)$ is a free $\cA(\ula^\svee)$-module, and hence $R(N_1, \hdots, N_k) \T_{\cA(\ula^\svee)} \cA'$ is a free $\cA'$-module. Therefore, we obtain \eqref{R with base changed to one variable} as an isomorphism of vector spaces. 
		We note that the fibers of the left and right hand sides of \eqref{R with base changed to one variable} at a point $c \in \bC$ are isomorphic as  $\g[t]$-modules to $(N_1 \ast \hdots \ast N_k)(c)$.
		It follows that both sides of \eqref{R with base changed to one variable} are isomorphic as $\g[t]$-modules. In particular, there is a surjection 
		\begin{equation*}
		\big( R(N_1, \hdots, N_k) \T_{\cA(\ula^\svee)} \bC_0 \big) [z] \twoheadleftarrow R(N_1, \hdots, N_k).
		\end{equation*}
		
		Using it, we obtain:
		\begin{multline*}
		  R(R(N_1, \hdots, N_k) \T_{\cA(\ula^\svee)} \bC_0, M_1, \hdots, M_m )
                  \T_{\cA(\la_1^\svee + \cdots + \la_k^\svee, \umu^\svee)} \bC_0 \\
		\simeq \Big( \big( R(N_1, \hdots, N_k) \T_{\cA(\ula^\svee)} \bC_0 \big)[z] \odot R(M_1, \hdots, M_m ) \Big) \T_{\cA(\la_1^\svee + \cdots + \la_k^\svee, \umu^\svee)} \bC_0 \\
		\twoheadleftarrow \Big( R(N_1, \hdots, N_k) \odot R(M_1, \hdots, M_m ) \Big) \T_{\cA(\ula^\svee, \umu^\svee)} \bC_0 \\
		\simeq R(N_1, \hdots, N_k, M_1, \hdots, M_m) \T_{\cA(\ula^\svee, \umu^\svee)} \bC_0.
		\end{multline*}
		Comparing the dimensions of the leftmost and the rightmost terms, we obtain part c) of the Proposition.
	\end{proof}
	
	\begin{cor}[From the proof of Proposition \ref{fusion is associative}]
		Suppose $M_1, \hdots, M_k$ are cyclic graded $\fg[t]$ modules such that 
		\[
		R(M_1, \hdots , M_k) \T_{\cA(\ula^\svee)} \bC_0 \simeq M_1 \ast \hdots \ast M_k.
		\]
		Let $\bc = (\underbrace{c_1, \hdots, c_1}_{i_1}, \hdots, \underbrace{c_n, \hdots, c_n}_{i_n}) \in \bC^k$ be such that $c_i \neq c_j$ for $i \neq j$. Then
		\begin{multline*}
		R(M_1, \hdots , M_k) \T_{\cA(\ula^\svee)} \bC_\bc
		\\ \simeq \big( M_1 \ast \hdots \ast M_{i_1} \big) (c_1) \ast \hdots \ast \big( M_{i_1 + \hdots + i_{n - 1} + 1} \ast \hdots \ast M_k \big) (c_n).
		\end{multline*}
	\end{cor}

	\section{Key objects of the paper}
	
	Simple Lie algebras: 
	
	\medskip
	
	\noindent $\fg$ -- simple Lie algebra of rank $r$ with Cartan decomposition $\fg=\fn_+ \oplus\fh\oplus\fn_-$;\\
	\noindent $G^\sic$ (resp.\ $G^\ad$) -- simply connected (resp.\ adjoint) complex Lie group of $\fg$;\\
	\noindent  $\al^\svee_1,\dots,\al^\svee_r$ -- simple roots, $\om^\svee_1,\dots,\om^\svee_r$ -- fundamental weights;\\
	\noindent  $\al_1,\dots,\al_r$ -- simple coroots, $\om_1,\dots,\om_r$ -- fundamental coweights;\\
	\noindent $P=\bigoplus_{i=1}^r \bZ \om_i \supset \bigoplus_{i=1}^r \bZ_{\ge 0} \om_i = P_+$ coweight lattice and its dominant cone;\\
	\noindent $P^\vee=\bigoplus_{i=1}^r \bZ \omega^\svee_i \supset \bigoplus_{i=1}^r
	\bZ_{\ge 0} \omega^\svee_i = P^\vee_+$ weight lattice and its dominant cone;\\
	\noindent $\iota\colon P\to P^\vee$ -- the linear map from the coweight lattice to the weight
	lattice corresponding to the minimal invariant even bilinear form on the coroot lattice
	(``level 1'');\\
	\noindent for $\la=\sum_{i=1}^r m_i\om_i\in P_+$ we let $|\la|=\sum_{i=1}^r m_i$;\\
	\noindent $V_{\la^\svee}$ -- irreducible $\fg$-module with highest weight $\la^\svee\in P^\vee_+$.
	
	\medskip
	
	Current and affine algebras:
	
	\medskip
	\noindent $\fg[t]=\fg\T\bC[t]$ -- current algebra;\\
	\noindent $W_{\lambda^\svee}$, $\bW_{\lambda^\svee}$ ($\lambda^\svee\in P_+$) --
	local and global Weyl modules for $\fg[t]$;\\
	\noindent $S_\la=\times_{i=1}^r S_{m_i}$ -- symmetric group attached to $\la=\sum_{i = 1}^r m_i \om_i\in P_+$;\\
	\noindent $\odot$ -- cyclic product;\\
	\noindent $\fgh$ -- affine Kac-Moody Lie algebra;\\
	\noindent $W$ and $W^a$ -- finite Weyl group and extended affine Weyl group;\\
	\noindent $D_{\ell,\la}$ ($\ell\in\bZ_{\ge 1}$, $\la\in P_+$) -- level $\ell$ weight $\ell\iota(\la)$ affine Demazure module;\\
	\noindent $\ula=(\la_1,\dots,\la_k)$ -- collection of integral dominant coweights;\\
	\noindent $\bD(\ell,\ula) \simeq R(D(\ell,\la_1),\dots,D(\ell,\la_k))$ -- global Demazure module;\\
	\noindent $\bD_{\ell,\la} \simeq \bD(\ell,  \underbrace{\om_1, \hdots , \om_1}_{m_1}, \hdots , \underbrace{\om_r, \hdots , \om_r}_{m_r})$,
	where $\la=\sum_{i=1}^r m_i\om_i$;\\
	\noindent $\cA(\ula)$ -- highest weight algebra of $\bD(\ell,\ula)$;\\
	\noindent $\cA_\la \simeq \cA(\underbrace{\om_1, \hdots , \om_1}_{m_1}, \hdots , \underbrace{\om_r, \hdots , \om_r}_{m_r}) $ -- 
	highest weight algebra of $\bD_{\ell,\la}$;\\
	\noindent $M^\vee={\rm Hom}_{\cA(\ula)} (M,\cA(\ula))$ -- $\cA(\ula)$-dual of an $\cA(\ula)$-module $M$.
	
	\medskip
	
	Geometry:
	
	\medskip
	\noindent $\ol\Gr{}^\la\subset \bP(D(1,\la))$ -- spherical affine Schubert variety; \\
	\noindent $\bA^\lambda= \Spec(\cA_\la)$ -- colored configuration space on the affine line;\\
	\noindent $\bA^{(\ula)}= \Spec(\cA(\ula))$ -- closure of a diagonal stratification stratum
	in a colored configuration space on the affine line;\\
	\noindent $\Lambda_0^\vee$ -- basic level one integrable affine weight;\\
	\noindent $\Lambda_0^\vee,\Lambda_1^\vee,\dots,\Lambda_m^\vee$ -- all level one integrable affine weights;\\
	\noindent $\Gr:=\Gr_{G^\ad}=G^\ad(\BC(\!(t)\!))/G^\ad(\BC[\![t]\!])$ -- affine Grassmannian of
	$G^\ad$;\\
	\noindent $\Gr\simeq \sqcup_{i=0}^m \Gr(\Lambda_i^\vee)$ -- decomposition into irreducible components;\\
	\noindent $\Gr_{\bA^k}$ -- Beilinson-Drinfeld Grassmannian over $\bA^k$;\\
	\noindent $\Gr^{\ula}\subset \Gr_{\BA^k}$ -- Beilinson-Drinfeld spherical Schubert variety;\\
	\noindent $\Gr^{(\ula)}$ -- partially symmetrized Beilinson-Drinfeld spherical Schubert variety over $\bA^{(\ula)}$;\\
	\noindent $\uG(k)$ -- group scheme acting on the Beilinson-Drinfeld Grassmannian;\\  
	\noindent $\cL$ -- very ample determinant line bundle;\\
	\noindent $\cD(\ell,\ula)$ -- locally free sheaf on $\bA^{(\ula)}$, corresponding to the free $\cA(\ula)$-module $\bD(\ell,\ula)$.

\end{document}